\documentclass[10pt]{amsart}

\makeatletter
\def\LaTeX{\leavevmode L\raise.42ex
    \hbox{\kern-.3em\size{\sfsize}{0pt}\selectfont A}\kern-.15em\TeX}
\makeatother

\newcommand{\BibTeX}{{\rm B\kern-.05em{\sc
          i\kern-.025emb}\kern-.08em\TeX}}

\makeatletter
\label{e:dispaa}
\label{e:dispau}
\label{e:dispav}
\label{e:dispaw}
\label{e:dispax}
\makeatother

\newtheorem{theorem}{Theorem}

\newtheorem{proposition}[theorem]{Proposition}
\newtheorem{lemma}[theorem]{Lemma}

\numberwithin{theorem}{section}
\numberwithin{equation}{section}

\usepackage{latexsym}
\usepackage{amsmath}
\usepackage{amsfonts}
\usepackage{amssymb}
\usepackage{amsthm}
\usepackage{cite}

\newcommand{\R}{\mathbb{R}}

\newcommand{\oO}{\overline{\Omega}}

\newcommand{\N}{\mathbb{N}}

\newcommand{\cA}{{\mathcal A}}   
   
\newcommand{\cC}{{\mathcal C}}

\newcommand{\cL}{{\mathcal L}}

\newcommand{\cS}{{\mathcal S}}

\newcommand{\mR}{{\mathbf R}}      
\newcommand{\mK}{{\mathbf K}}

\newcommand{\dist}{{\rm dist}}
\newcommand{\supp}{{\rm supp}}

\newcommand{\eps}{\varepsilon}

\renewcommand{\phi}{\varphi}

\def\eps{\varepsilon}

\begin{document}

\title[Branch continuation for the nonlinear Schr{\"o}dinger equation]{Branch continuation inside 
the essential spectrum for the nonlinear Schr\"odinger equation}
\author{Gilles Ev\'equoz}
\email{evequoz@math.uni-frankfurt.de}
\author{Tobias Weth}
\email{weth@math.uni-frankfurt.de}
\address{ Institut f\"ur Mathematik, Johann Wolfgang Goethe-Universit\"at,
Robert-Mayer-Str. 10, 60054 Frankfurt am Main, Germany}

\begin{abstract}
We consider the nonlinear stationary Schr{\"o}dinger equation   
\begin{equation*}
 -\Delta u -\lambda u=  Q(x)|u|^{p-2}u,  \qquad \text{in }\mathbb{R}^N
\end{equation*}
in the case where $N \ge 3$, $p$ is a superlinear, subcritical exponent, $Q$ is a bounded, nonnegative and nontrivial weight function with compact support in $\mathbb{R}^N$ and $\lambda \in \mathbb{R}$ is a parameter. Under further restrictions either on the exponent $p$ or on the shape of $Q$, we establish the existence of a continuous branch $\mathcal{C}$ of nontrivial solutions to this equation which intersects $\{\lambda \} \times L^{s}(\mathbb{R}^N)$ for every $\lambda \in (-\infty, \lambda_Q)$ and $s> \frac{2N}{N-1}$. Here $\lambda_Q>0$ is an explicit positive constant which only depends on $N$ and $\text{diam}(\text{supp }Q)$. In particular, the set of values $\lambda$ along the branch enters the essential spectrum of the operator $-\Delta$.  
\end{abstract}

\dedicatory{Dedicated to Paul H. Rabinowitz with admiration and appreciation}

\maketitle

\section{Introduction and main result}
\label{sec:intr-main-result} 
The present paper is concerned with nonlinear Schr{\"o}dinger equations of the type 
\begin{equation}
  \label{eq:eq-general-schr}
 -\Delta u + V(x)u - Q(x)|u|^{p-2}u = \lambda u, \qquad x \in \R^N,
\end{equation}
where $p>2$, $V \in L^\infty(\R^N)$ is a linear external potential, $Q \in L^\infty(\R^N)$ is a nonnegative weight function and $\lambda \in \R$ is a parameter. Within the last four decades, there has been a huge amount of work on equations of this form, whereas the majority of papers is devoted to the existence and multiplicity of weak solutions $u \in H^1(\R^N)$ for fixed $\lambda$ satisfying the assumption 
\begin{equation}
  \label{eq:main-assumption}
\lambda < \lambda_{ess}:= \inf \sigma_{ess}(-\Delta + V),  
\end{equation}
where $\sigma_{ess}(-\Delta + V)$ denotes the essential spectrum of the Schr{\"o}dinger operator $-\Delta +V$. In particular, we wish to point out the classical papers \cite{rabinowitz:92,berestycki-lions83,berestycki-lions83b,floer.weinstein:86,ding.ni:86}. 
Our present paper is motivated by a somewhat different viewpoint focusing on properties of $\lambda$-dependent solution families for $\lambda$ close to $\lambda_{ess}$. Starting with the pioneering works of Stuart \cite{stuart:82,stuart:80,stuart:80-0}, the role of $\lambda_{ess}$ as a bifurcation point for solutions of \eqref{eq:eq-general-schr} has been studied extensively in the past.  For $s \ge 1$, we will call the value $\lambda_{ess} \in \R$ an $L^{s}$-{\em bifurcation point} for \eqref{eq:eq-general-schr} if there is a sequence $(\lambda_k,u_k)_k$ in $(-\infty,\lambda_{ess}) \times L^{s}(\R^N)$ such that $u_k$ solves \eqref{eq:eq-general-schr} with $\lambda=\lambda_k$ for every $k \in \N$ and $(\lambda_k, u_k) \to (\lambda_{ess},0)$ in $\R \times L^s(\R^N)$. Necessary and sufficient  conditions on $Q$, $p$ and $s$ for  
$L^s$-bifurcation at $\lambda_{ess}$ were established e.g. in 
\cite{AB98,stuart:88,Ro89,ZZ88,Ca95}. For further results, we refer the reader to the illuminating survey \cite{stuart97}. 

The most general results on bifurcation sequences converging to 
 $(\lambda_{ess},0) \in \R \times L^s(\R^N)$ were 
obtained by variational methods. On the other hand, purely variational arguments  do not give rise to continua of solutions in $\R \times L^s(\R^N)$ emanating from $(\lambda_{ess},0)$.
Consequently, several different arguments have been applied in the past  to construct continua of this type. We mention in particular bounded domain approximation \cite{toland:1984}, shooting methods \cite{weth05}, rescaling arguments combined with the implicit function theorem  \cite{stuart:85} and generalizations of the Poincar\'e-Melnikov method 
\cite{magnus:88, stuart:89,AB98,AB98.1,BP04}. 
These methods rely on rather specific additional assumptions, and no general abstract method is available for the study of continuous bifurcation 
branches emanating from points in the essential spectrum 
of a linear operator. In contrast, the global branching properties of the pure point 
spectrum of an operator with compact resolvent have been analyzed in great detail by means of a
degree theoretic approach, which has been introduced by Rabinowitz in his classical paper \cite{rabinowitz71} and has been developed further e.g. by Dancer in \cite{dancer:73,dancer:2002}. 

The present paper is concerned with continua of solutions in the case where $N \ge 3$, $V \equiv 0$, $Q \in L^\infty(\R^N)$ is a nonnegative function with compact support
and $p$ is subcritical, i.e., $2<p<2^*:= \frac{2N}{N-2}$. In this case, \eqref{eq:eq-general-schr} reduces to 
\begin{equation}
  \label{eq:eq-general-schr-1}
 -\Delta u -\lambda u=  Q(x)|u|^{p-2}u,  \qquad x \in \R^N,
\end{equation}
and it follows from \cite[p.~526--527]{stuart:88} that $\lambda_{ess}= 0$ is no $L^s(\R^N)$-bifurcation point for \eqref{eq:eq-general-schr-1} for $s \ge \max\{1,\frac{N(p-2)}{2}\}$.  Nevertheless, we shall see that, for a range of subcritical exponents $p$, there exists solution branches which can be continued to values $\lambda>0$, i.e., to values inside $\sigma_{ess}(-\Delta)$. For $\lambda>0$, these solutions are not in $H^1(\R^N)$, but they remain in $L^s(\R^N)$ for $s>\frac{2N}{N-1}$, and the branches will be continuous with respect to the $L^s$-norm.

To state our main result, we introduce the following notation. If $E$ is a normed vector space and $I \subset \R$ is an interval, a connected subset $\cC$ of $I \times E$ will be called a 
{\em global branch in $I \times E$} if $\cC$ intersects $\{t\} \times E$ nontrivially for every $t \in I$. 
For $1\leq s\leq \infty$ we denote by $L^s_c(\R^N)$ the (equivalence class of) functions in $L^s(\R^N)$ having compact support in $\R^N$. Moreover, for $N \ge 3$, we denote by $2^*:= \frac{2N}{N-2}$ the critical Sobolev exponent. 
In our main result, the following (alternative) assumptions will be used. 
\begin{itemize}
\item[(A1)] $2<p< \frac{2^*}{2}+1 = \frac{2(N-1)}{N-2}$, and $Q \in L^\infty_c(\R^N) \setminus \{0\}$ is a nonnegative function.
\item[(A2)] $2<p< 2^*$, and  $Q \in L^\infty_c(\R^N)$ has the form 
\begin{equation}\label{eqn:Q_dist-assumption}
Q(x)= \beta(x) \text{dist}(x, \R^N \! \setminus \!\Omega)^\alpha \qquad \text{ for all $x\in \R^N$,}
\end{equation}
with some $\alpha >0$, a bounded open set $\Omega$ of class $C^1$ and a positive function $\beta \in  C(\R^N)$. 
\end{itemize}

We remark that in assumption (A2) it clearly suffices to assume that $\beta$ is continuous on $\overline \Omega$, since then it can be extended arbitrarily to a function $\beta \in C(\R^N)$ and \eqref{eqn:Q_dist-assumption} still holds. 

\begin{theorem}
\label{main-theorem}
Suppose that  $N \ge 3$, and that $(A1)$ or $(A2)$ holds. Then there exists $\lambda_Q>0$ -- depending only on $N$ and $\text{diam}(\text{supp }Q)$ -- such that, for any $\Lambda^* \in (0,\lambda_Q)$, \eqref{eq:eq-general-schr-1} admits  a global branch $\cC$ of nontrivial strong solutions in $(-\infty,\Lambda^*] \times L^s(\R^N)$ for all $s>\frac{2N}{N-1}$. Moreover, if $(\lambda,u) \in \cC$, then we have:
\begin{itemize}
\item[(i)] $u \in C^{1,\gamma}_{loc}(\R^N)$ for all $\gamma \in (0,1)$;  
\item[(ii)] If $\lambda \le 0$, then $u$ is positive on $\R^N$ ;
\item[(iii)] If $\lambda > 0$, then $u$ is positive on $\text{supp }Q$ and changes sign in $\R^N \setminus \text{supp }Q$.
\end{itemize}
Explicitly, $\lambda_Q$ is given by 
$$
\lambda_Q = \left(\frac{y_{\frac{N-2}{2}}^{(1)}}{\text{diam}(\text{supp }Q)}\right)^2,
$$
where $y_{\frac{N-2}{2}}^{(1)}$ is the first positive zero of $Y_{\frac{N-2}{2}}$, the Bessel function of the second kind of order $\frac{N-2}{2}$. 
\end{theorem}

We remark that the mere existence of solutions of \eqref{eq:eq-general-schr-1} for fixed $\lambda>0$ has been proved in the recent papers \cite{evequoz-weth-dual,evequoz-weth14}, see also \cite{evequoz:2015-1}. In the case $\lambda>0$, \eqref{eq:eq-general-schr-1} is usually called the nonlinear Helmholtz equation, and it is well known that this equation does not admit solutions in $L^2(\R^N)$. Consequently, the variational structure of \eqref{eq:eq-general-schr-1} is only formal in this case, since the associated energy functional contains the square of the $L^2(\R^N)$-norm. In the papers  \cite{evequoz-weth-dual,evequoz-weth14}, different methods have been used to reformulate the problem and to recover a valid variational setting. In \cite{evequoz-weth14}, for compactly supported $Q$ (and a more general class of superlinear
nonlinearities $f$ compactly supported in space), we used a Dirichlet-to-Neumann map associated to the exterior problem for the linear
Helmholtz equation to recast the problem in a variational setting in the space $H^1(B_R(0))$ for some $R>0$. Moreover, in \cite{evequoz-weth-dual} we have 
set up a dual variational framework for equation \eqref{eq:eq-general-schr-1} in the case $\lambda>0$ which also allows to study noncompactly supported weight functions $Q$. More precisely, using resolvent estimates for the linear Helmholtz operator, we have reformulated \eqref{eq:eq-general-schr-1} as a (generalized) integral equation in $v:= Q^{\frac{1}{p'}}|u|^{p-2}u$ which has a variational structure in $L^{p'}(\R^N)$, where $p':= \frac{p}{p-1}$ denotes the
conjugate exponent of $p$. 

In the present paper, we will also use a reformulation of \eqref{eq:eq-general-schr-1} as an integral equation, but instead of variational methods we will apply fixed point index calculations and the Leray-Schauder continuation principle to the reformulated problem in order to construct global solution branches. The main difficulty within this approach is the existence of a priori bounds. It clearly follows from the multiplicity results in \cite{evequoz-weth-dual,evequoz-weth14} that the set of all solutions of \eqref{eq:eq-general-schr-1} is unbounded in $L^p(\R^N)$ for fixed $\lambda>0$. Hence we need to restrict our attention to values $\lambda$ sufficiently close to $0$ and 
to a subclass of `admissible solutions'  which are positive on $\text{supp }Q$. We then use two different methods in order to derive a priori bounds. The first one is based on testing the equation with the first eigenfunction of an associated weighted eigenvalue problem and using bootstrap arguments. In the more classical context of semilinear Dirichlet problems on bounded domains, this method goes back to \cite{brezis-turner:77}. This method only applies for a restricted range of exponents as considered in assumption $(A1)$, but it does not require further conditions on the shape of $Q$. The second method is a contradiction argument based on a rescaling procedure and on Liouville type theorems for the corresponding limit problems.  
In the context of semilinear Dirichlet problems on bounded domains, this argument goes back to \cite{gidas.spruck81} and has also been applied to 
indefinite problems, see e.g. \cite{berestycki-cd-nirenberg94,amann-lopez98,ramosetal98}. Our approach is based on a relatively recent Liouville theorem in \cite{du.li:05} which, in the context of the present problem, leads to the assumption in $(A2)$ on the shape of $Q$ near the boundary of its support.  On the other hand, in contrast to the first method, it allows to consider the full range of subcritical exponents $p$.

We wish to point out three questions left open by Theorem~\ref{main-theorem}. 
First, we do not know whether the global branches provided in Theorem~\ref{main-theorem} can be continued to values $\lambda \ge \lambda_Q$. Second, in the case $N=2$ we cannot expect a result like Theorem~\ref{main-theorem} to hold, since branches of solutions which are positive on $\supp \, Q$ cannot intersect $\{0\} \times L^s(\R^N)$ for any $s \in [1,\infty]$.  Indeed, for $\lambda=0$, every solution of \eqref{eq:eq-general-schr-1} which is positive on $\supp \, Q$ is superharmonic, and there do not exist bounded nonconstant superharmonic functions in the case $N=2$. 
It is therefore an open question, which function space is most suitable to analyze solution branches in the case $N=2$. 
 The third open problem is concerned with the case where $Q$ is not compactly supported. At this point we have no idea how to obtain suitable a priori bounds in this case, and the local positivity property defining the subclass of admissible solutions also needs to be defined in a different way.

Finally, we wish to mention related work of  Heinz \cite{He86} who constructed global continuous branches of radial solutions with a prescribed number of zeros for \eqref{eq:eq-general-schr-1} with $\lambda>0$, in the case where $Q$ is negative and of the form $Q(x)=-e^{f(|x|)}$ with a positive and nondecreasing function 
$f$. We note that equation \eqref{eq:eq-general-schr-1} is usually called sublinear when $Q$ is negative. For a similar 
one-dimensional sublinear problem, 
a continuous branch of positive solutions was already exhibited by K\"upper in \cite{Ku79}. For more general sublinear equations, positive solution branches have also been detected by Giacomoni in \cite{Gi98}.

The paper is organized as follows. In Section~\ref{sec:resolv}, we collect useful estimates for the fundamental solutions of the operators $-\Delta- \lambda$, $\lambda \in \R$ and the corresponding resolvent operators. In Section~\ref{sec:priori-bounds-solut}, we prove some upper and lower $L^{\infty}$-estimates which apply to nontrivial solutions $u \in L^{2^*}(\R^N)$ of \eqref{eq:eq-general-schr-1} and which are locally uniform in $\lambda \in \R$. In Section~\ref{sec:priori-bounds-nonn}, we reformulate our problem as an integral equation in the space of continuous functions on $\supp \, Q$, and we establish a priori bounds for positive solutions of this problem in the case where $\lambda< \lambda_Q$ and one of the assumptions $(A1)$ or $(A2)$ is satisfied. In Section~\ref{sec:glob-branch-posit},  we then complete the proof of Theorem~\ref{main-theorem} by means of fixed point index calculations and the Leray-Schauder continuation principle. Finally, in the appendix we state a result from general topology which is used in Theorem~\ref{main-theorem} in order to pass to unbounded solution branches.

\section{Notation and preliminary results}
\label{sec:resolv}

Throughout the paper, we assume that $N \ge 3$, and we let $2^*:= \frac{2N}{N-2}$ denote the critical Sobolev exponent. For $\lambda \in \R$, 
we consider the following real-valued fundamental solution of the operator $-\Delta -\lambda$ on $\R^N$, see e.g. \cite{gelfand}:
$$
\Psi_\lambda(x)
=\left\{\begin{array}{ll} \frac1{2\pi} \left(\frac{\sqrt{|\lambda|}}{2\pi|x|}\right)^{\frac{N-2}{2}} K_{\frac{N-2}{2}}\left(\sqrt{|\lambda|}\ |x|\right),& \lambda<0 \\ 
\\ \frac{\Gamma\left(\frac{N-2}{2}\right)}{4\pi^{\frac{N}{2}}} |x|^{2-N}, & \lambda=0 \\ \\ 
-\frac{1}{4}\left(\frac{\sqrt{\lambda}}{2\pi |x|}\right)^{\frac{N-2}{2}} Y_{\frac{N-2}{2}}\left(\sqrt{\lambda}\ |x|\right),& \lambda>0. \end{array}\right.
$$
Here $Y_\nu$ denotes the Bessel function of the second kind and $K_{\nu}$ denotes the Macdonald function (or Bessel function of the third kind) of order $\nu$. In the following, we let $\cS$ denote the usual Schwartz space of rapidly decreasing functions.  The functions $\Psi_\lambda$ generate linear operators 
\begin{equation}
  \label{eq:mR-lambda-prelim}
\cS \to C^\infty(\R^N), \qquad f \mapsto  \Psi_\lambda \ast f, 
\end{equation}
where $\ast$ denotes the usual convolution. For $f\in\cS$, the function $u:= \Psi_\lambda\ast f\in C^\infty(\R^N)$ solves the inhomogeneous
equation $-\Delta u-\lambda u=f$ and satisfies the decay estimate $|u(x)|=O\left(|x|^{\frac{1-N}{2}}\right)$. We point out that, for $\lambda>0$, $\Psi_\lambda$ is the real part of the function $$
x \mapsto \frac{i}{4}\left(\frac{\sqrt{\lambda}}{2\pi |x|}\right)^{\frac{N-2}{2}} H^{(1)}_{\frac{N-2}{2}}\bigl(\sqrt{\lambda}\ |x|\bigr)
$$ 
which is the usual fundamental solution of the Helmholtz operator associated with Sommerfeld's (outward) radiation condition,  see e.g. \cite[Section 2]{evequoz-weth-dual}. Here $H^{(1)}_{\frac{N-2}{2}}$ denotes the Hankel function of the first kind of order $\frac{N-2}{2}$. In contrast, $\Psi_\lambda$ should be seen as the fundamental solution associated with standing wave solutions of the linear (inhomogeneous) Helmholtz equation.

We need some estimates for the functions $\Psi_\lambda$. If $\nu\geq 0$, we denote by $y_\nu^{(1)}$ the first positive
zero of the function $Y_\nu$. Moreover, we put 
\begin{equation}
  \label{eq:def-gamma-N}
\gamma_N:=\left\{\begin{array}{ll} 1, & N=3, \\ 
-\frac{\pi}{\Gamma\left(\frac{N-2}{2}\right)} \left(\frac12 y_{\frac{N-4}{2}}^{(1)}\right)^{\frac{N-2}{2}}
Y_{\frac{N-2}{2}}\left(y_{\frac{N-4}{2}}^{(1)}\right), & N\geq 4.
\end{array}\right.
\end{equation}

\begin{lemma} \label{lem:bessel}$ $

(i) For all $x \in \R^N \setminus \{0\}$, the function $\R \to \R$, $\lambda\mapsto \Psi_\lambda(x)$ is continuous.

(ii) For $\lambda\leq0$ and $x\in\R^N$ we have $0<\Psi_\lambda(x)\leq \Psi_0(x)$.

(iii) For $\lambda> 0$ and $x\in\R^N$ with $\sqrt{\lambda}|x|<y_{\frac{N-2}{2}}^{(1)}$ we have $0<\Psi_\lambda(x)\leq \gamma_N \Psi_0(x)$.  

(iv) For every $\Lambda_*, {\Lambda^*},r_0 >0$ satisfying $\sqrt{{\Lambda^*}}r_0<y^{(1)}_{\frac{N-2}{2}}$ 
there exists $\eps_0=\eps_0(\Lambda_*,{\Lambda^*},r_0)>0$ such that 
$$
\Psi_\lambda(x)\geq \eps_0\Psi_0(x) \qquad \text{for $x \in \R^N$ with $|x|\le r_0$ and $\lambda\in[-\Lambda_*,{\Lambda^*}]$.}
$$

(v) There exists a constant $\zeta_N\geq 0$ such that 
$$
|\Psi_\lambda(x)|\leq 
\gamma_N\Psi_0(x) +\zeta_N\ \lambda^{\frac{N-3}{4}} |x|^{\frac{1-N}{2}}\qquad \text{for all $x\in\R^N$, $\lambda>0$.}
$$

If $N=3$ the estimate holds with $\zeta_N=0$.
\end{lemma}
\begin{proof}
For $r>0$, we consider the functions
$$
\psi_r: \R \to \R, \qquad   \psi_r(\lambda):=\left\{
\begin{aligned}
& \frac2{\Gamma(\frac{N-2}{2})}\Bigl(\frac{\sqrt{|\lambda|}\ r}{2}\Bigr)^{\frac{N-2}{2}} K_{\frac{N-2}{2}}\left(\sqrt{|\lambda|}\ r\right),&&\qquad \lambda<0, \\ 
&-\frac{\pi}{\Gamma(\frac{N-2}{2})}\Bigl(\frac{\sqrt{\lambda}\ r}{2}\Bigr)^{\frac{N-2}{2}} Y_{\frac{N-2}{2}}\left(\sqrt{\lambda}\ r\right),&&\qquad \lambda>0,\\
& 1,&&\qquad \lambda=0.
  \end{aligned}
\right.
$$
The asymptotics of the Bessel functions $K_\nu$ and $Y_\nu$ for small positive argument (see, e.g., \cite[Eq. (5.16.2) and (5.16.4)]{lebedev}) imply that
$$
\lim_{\lambda \to 0} \psi_r(\lambda)= 1= \psi_r(0) \qquad \text{for all $r>0$,}
$$
so the functions $\psi_r$ are continuous. Moreover, by definition we have 
\begin{equation}
  \label{eq:fundamental-psi-rel}
\Psi_\lambda(x)=\psi_{|x|}(\lambda)\Psi_0(x) \qquad \text{for $x \in \R^N \setminus \{0\}$ and $\lambda \in \R$,}  
\end{equation}
and thus we obtain assertion (i).

Next we note that the function $K_{\frac{N-2}{2}}$ is strictly positive on $(0,\infty)$ and the function $Y_{\frac{N-2}{2}}$ is strictly
negative in the range $(0,y_{\frac{N-2}{2}}^{(1)})$ (see \cite[p.~135--136]{lebedev}). Therefore, 
for $x \in \R^N \setminus \{0\}$ we have 
$$
\Psi_\lambda(x)>0 \qquad \text{if $\lambda\leq 0 \quad$ or $\quad \lambda>0$ and $\sqrt{\lambda}|x|<y_{\frac{N-2}{2}}^{(1)}$.}
$$ 
This proves the first inequality in (ii) and (iii). 
To prove the other inequality, we first consider the case $N=3$. 
In this case, using \cite[Eq. (5.8.4) and (5.8.5)]{lebedev}, we find that 
\begin{equation}\label{eqn:psi3}
\psi_r(\lambda)= e^{-\sqrt{|\lambda|}r} \quad \text{for $\lambda<0$} \qquad \text{and}\qquad 
\psi_r(\lambda)= \cos\left(\sqrt{\lambda}\ r\right) \quad \text{for $\lambda>0.$}
\end{equation}
In particular, $\psi_r(\lambda)\leq 1=\psi_r(0)$ for $\lambda\in\R$, $r>0$, 
showing that $\Psi_\lambda(x)\leq\Psi_0(x)$ for all $\lambda$ and $x$.
This concludes the proof of (ii) and (iii) in dimension $N=3$. 
When $N\geq 4$, we differentiate $\psi_r$ and obtain
\begin{equation}\label{eqn:deriv_psi}
\frac{d}{d\sqrt{|\lambda|}}\psi_r(\lambda)=\left\{\begin{array}{ll}
-\frac{2r}{\Gamma(\frac{N-2}{2})}\left(\frac{\sqrt{|\lambda|}\ r}{2}\right)^{\frac{N-2}{2}} K_{\frac{N-4}{2}}\left(\sqrt{|\lambda|}\ r\right),&\qquad \lambda<0, \\ 
\\ 
-\frac{\pi r}{\Gamma(\frac{N-2}{2})}\left(\frac{\sqrt{\lambda}\ r}{2}\right)^{\frac{N-2}{2}} Y_{\frac{N-4}{2}}\left(\sqrt{\lambda}\ r\right),&\qquad \lambda>0. 
\end{array}\right.
\end{equation}
Using again the positivity of $K_\nu$ for $\nu\geq 0$, we deduce that $\psi_r(\lambda)\leq\psi_r(0)=1$, 
and therefore $\Psi_\lambda(x)\leq\Psi_0(x)$, for all $\lambda\leq 0$ and all $r>0$, $x\in\R^N$. 
For $\lambda>0$, we use the fact that the zeros of $Y_{\frac{N-4}{2}}$
and $Y_{\frac{N-2}{2}}$ are interlaced with $y_{\frac{N-4}{2}}^{(1)}<y_{\frac{N-2}{2}}^{(1)}$ (see \cite[Sec. 15.22]{watson}), whereas $Y_{\frac{N-4}{2}}$ is negative in $(0,y_{\frac{N-4}{2}}^{(1)})$ and positive in $(y_{\frac{N-4}{2}}^{(1)},y_{\frac{N-2}{2}}^{(1)})$. Hence \eqref{eqn:deriv_psi} gives 
\begin{equation}
  \label{eq:gamma-N-est}
0 \le \psi_r(\lambda)\leq \psi_r\left(r^{-2}\left(y_{\frac{N-4}{2}}^{(1)}\right)^2\right)
=-\frac{\pi}{\Gamma(\frac{N-2}{2})}\left({\textstyle\frac12}y_{\frac{N-4}{2}}^{(1)}\right)^{\frac{N-2}{2}} 
Y_{\frac{N-2}{2}}\left(y_{\frac{N-4}{2}}^{(1)}\right)=\gamma_N 
\end{equation}
for $\lambda>0$ with $\sqrt{\lambda}r\leq y_{\frac{N-2}{2}}^{(1)}$. Combining this with \eqref{eq:fundamental-psi-rel}, we infer that 
$$
\Psi_\lambda(x)\leq \gamma_N\Psi_0(x) \qquad \text{for $x \in \R^N$ and $\lambda>0$ 
with $0< \sqrt{\lambda}|x| \leq y_{\frac{N-2}{2}}^{(1)}$.}
$$ 
Hence, assertions (ii) and (iii) also hold in the case $N\geq 4$.

Let us now fix $\Lambda_*, {\Lambda^*}, r_0>0$ such that $\sqrt{{\Lambda^*}}\ r_0<y_{\frac{N-2}{2}}^{(1)}$.
As a consequence of \eqref{eqn:psi3}, \eqref{eqn:deriv_psi} and the fact that  
$$
\frac{d}{dr}\psi_r(\lambda)=\frac{\sqrt{|\lambda|}}{r}\frac{d}{d\sqrt{|\lambda|}}\psi_r(\lambda)<0 \qquad \text{for $\lambda < 0$, $r>0$,}
$$
we see that 
\begin{equation}
  \label{eq:r-Lambda-star-1}
\psi_r(\lambda)\geq \psi_r(-\Lambda_*)
\ge \psi_{r_0}(-\Lambda_*)>0\qquad \text{for $0\geq\lambda\geq-\Lambda_*$ and $0<r\leq r_0$.}
\end{equation}
On the other hand, if $\lambda>0$, then \eqref{eqn:psi3} and \eqref{eqn:deriv_psi} imply that 
\begin{equation}
  \label{eq:r-Lambda-star-2}
\psi_r(\lambda)\geq\min\{1,\psi_{r_0}({\Lambda^*})\}>0 \qquad \text{for $\lambda\leq{\Lambda^*}$ and $0<r\leq r_0$,}
\end{equation}
 since $\psi_r(0)=1$ and $\sqrt{\lambda}\ r\leq\sqrt{{\Lambda^*}}\ r_0<y_{\frac{N-2}{2}}^{(1)}$.
Setting $\eps_0=\min\{1,\psi_{r_0}(-\Lambda_*),\psi_{r_0}({\Lambda^*})\}$, we then deduce from 
\eqref{eq:fundamental-psi-rel},~\eqref{eq:r-Lambda-star-1} and \eqref{eq:r-Lambda-star-2} that
$$
\Psi_\lambda(x)\geq \eps_0\Psi_0(x)\qquad \text{for $x \in \R^N$ with $0<|x|\leq r_0$ and $\lambda\in[-\Lambda_*,{\Lambda^*}]$},
$$ 
thus proving (iv).

It remains to prove (v). In the case where $N=3$, \eqref{eq:fundamental-psi-rel} and \eqref{eqn:psi3} readily give 
$$
|\Psi_\lambda(x)|\leq \Psi_0(x) \qquad \text{for $\lambda>0$ and $x \in \R^N \setminus \{0\}$.}
$$ 
In the case where $N\geq 4$, we let
$$
\chi(t):=-\frac{\pi}{\Gamma(\frac{N-2}{2})}\left(\frac{t}2\right)^{\frac{N-2}{2}}Y_{\frac{N-2}{2}}(t) \quad \text{for $t>0$}\qquad \text{and}\qquad \chi(0)=1.
$$
From our above considerations, we see that $\chi$ is continuous on $[0,\infty)$. Moreover, 
\eqref{eq:gamma-N-est} and the identity $\psi_r(\lambda)=\chi(\sqrt{\lambda}r)$, for $r>0$, $\lambda\geq0$ imply that 
$$
0\leq\chi(t)\leq\gamma_N \qquad \text{for $0\leq t\leq y_{\frac{N-2}{2}}^{(1)}$.}
$$
Furthermore, the asymptotics of $Y_{\frac{N-2}{2}}$ yield
$\limsup\limits_{t\to\infty}t^{-\frac{N-3}{2}}|\chi(t)|<\infty$ (see \cite[Eq. (5.16.2)]{lebedev}).
Therefore, we can find some constant $\tilde{\zeta}_N>0$ such that
$|\chi(t)|\leq \gamma_N+\tilde{\zeta}_N t^{\frac{N-3}{2}}$ for all $t\geq 0$.
The conclusion follows by setting $\zeta_N=\tilde{\zeta}_N\frac{4\pi^{\frac{N-2}{2}}}{\Gamma(\frac{N-2}{2})}$
and by noticing that $\Psi_\lambda(x)=\chi(\sqrt{\lambda}|x|)\Psi_0(x)$ for $x \in \R^N \setminus \{0\}$ and $\lambda\geq 0$.
\end{proof}

In the following, we collect some fundamental properties of the resolvent operators introduced in \eqref{eq:mR-lambda-prelim}. For this it will be useful to 
consider their extensions to continuous linear operators $L^{(2^\ast)'}(\R^N) \to L^{2^\ast}(\R^N)$. 

\begin{lemma} \label{lem:reg1bis-0}
For every $\lambda\in\R$, the map $\cS \to C^\infty(\R^N), \; f \mapsto  \Psi_\lambda \ast f$ extends to a continuous linear operator 
$$
\mR_\lambda: L^{(2^\ast)'}(\R^N) \to L^{2^\ast}(\R^N).
$$ 
\end{lemma}

\begin{proof}
By definition, $\Psi_0 \in L^{\frac{N}{N-2},w}(\R^N)$, the weak $L^{\frac{N}{N-2}}(\R^N)$-space.  For $\lambda \le 0$, we may therefore use 
Lemma~\ref{lem:bessel}(ii) and the weak Young inequality to see that
$$
\|\mR_\lambda(f)\|_{2^\ast}=\|\Psi_\lambda\ast f\|_{2^\ast}\leq \|\Psi_0\ast |f|\ \|_{2^\ast}
\leq\|\Psi_0\|_{\frac{N}{N-2},w} \|f\|_{(2^\ast)'} \qquad \text{for $f \in L^{(2^\ast)'}(\R^N)$,}
$$
where $\|\cdot\|_{\frac{N}{N-2},w}$ denotes the norm in $L^{\frac{N}{N-2},w}(\R^N)$. For $\lambda>0$, the result is a special case of a Theorem by Kenig, Ruiz and Sogge \cite[Theorem 2.3]{KRS87}.
\end{proof}

The next lemma is concerned with interior elliptic estimates for solutions of the inhomogeneous Helmholtz equation defined via the operator $\mR_\lambda$.

\begin{lemma} \label{lem:reg1bis}
Let $\lambda\in\R$, $f\in L^{(2^\ast)'}(\R^N)$ and $u:=\mR_\lambda(f)$. Then we have:
\begin{itemize}
\item[(a)]  $u \in W^{2,(2^\ast)'}_{\text{loc}}(\R^N)\cap L^{2^\ast}(\R^N)$ is 
a strong solution of $-\Delta u-\lambda u=f$ in $\R^N$.
\item[(b)] If, in addition, $f, u \in L^t_{\text{loc}}(\R^N)$ for some $t \in (1,\infty)$, then $u\in W^{2,t}_{\text{loc}}(\R^N)$.\\
Moreover, for given ${\Lambda_*}, \Lambda^*>0$ and $0<r<R$ there exists a constant $C>0$ depending only on $N,t,R,r,{\Lambda_*}$ and $\Lambda^*$ such that in case $-\Lambda_* \le \lambda \le \Lambda^*$ we have
$$
\|u\|_{W^{2,t}(B_r(x_0))}\leq C \left( \|u\|_{L^t(B_R(x_0))}+\|f\|_{L^t(B_R(x_0))}\right) \qquad \text{for every $x_0 \in \R^N$.}
$$
\end{itemize}
\end{lemma}

\begin{proof}
For $\lambda>0$ the assertions (a) and (b) follow directly from \cite[Proposition A.1]{evequoz-weth-dual}. Moreover, in the case $\lambda\leq 0$, we have $u:=\mR_\lambda(f) \in L^{2^\ast}(\R^N)$ by Lemma~\ref{lem:reg1bis-0}, and a standard argument shows that $u:=\mR_\lambda(f)$ solves the equation $-\Delta u-\lambda u=f$ in $\R^N$ in distributional sense. Therefore, exactly the same proof as in 
\cite[Proposition A.1]{evequoz-weth-dual} (based on the Calder\'on-Zygmund estimate) yields properties (a) and (b) also in this case.
\end{proof}

The next Lemma yields a uniform $L^p$ estimate
for solutions of the inhomogeneous Helmholtz equation 
with compactly supported right-hand side and $\lambda$ bounded from above.
\begin{lemma}\label{lem:reg2}
Let $r_0, r, {\Lambda^*} >0$, $2<p\leq 2^\ast$ and $1 \le q \le p$. Then there exists a constant $D=D(N,p,q,r_0,r,{\Lambda^*})>0$ such that for
$f\in L^{p'}_c(\R^N)$ with $\text{diam}(\text{supp }f)\leq r_0$, $\lambda\leq{\Lambda^*}$ and $x_0\in\R^N$ there holds
$$
\|\mR_\lambda(f)\|_{L^q(B_r(x_0))}\leq D \|f\|_{p'}.
$$
\end{lemma}
\begin{proof}
Since, by H\"older's inequality, 
$$
\|u\|_{L^q(B_r(x_0))}\leq |B_r(0)|^{\frac{1}{q}-\frac{1}{p}} \|u\|_{L^p(B_r(x_0))} \qquad \text{for $q \le p$, $r>0$, $x_0 \in \R^N$ and $u \in L^p(B_r(x_0))$,}
$$
it suffices to prove the estimate in the case $q=p$. We first note that, by Lemma~\ref{lem:bessel}(ii),(v), we have 
\begin{equation}
  \label{eq:combined-concl}
|\Psi_\lambda(x)|\leq \gamma_N\Psi_0(x)+\zeta_N ({\Lambda^*})^{\frac{N-3}{4}}|x|^{\frac{1-N}{2}},\quad\text{ for all $x\in\R^N,  \lambda\leq{\Lambda^*}$,}
\end{equation}
where $\zeta_N=0$ if $N=3$. Now, let $f\in L^{p'}_c(\R^N)$ with $\text{diam}(\text{supp }f)\leq r_0$ and choose $y_0\in\text{supp }f$. If $p<2^\ast$, Young's inequality gives 
\begin{equation}
  \label{eq:r-x-0-est-1}
\|u\|_{L^p(B_r(x_0))}=\|\Psi_\lambda\ast f\|_{L^p(B_r(x_0))}
\leq \|\Psi_\lambda\|_{L^{t}(B_{r+r_0}(x_0-y_0))} \|f\|_{p'} 
\end{equation}
for $r>0$ and $x_0\in\R^N$, where $1<t:=\frac{p}2<\frac{N}{N-2}$. Moreover, for $\rho>0$, $x_0\in\R^N$ we have
\begin{align*}
\left(\frac{4\pi^{\frac{N}{2}}}{\Gamma\left(\frac{N-2}{2}\right)}\right)^t\int_{B_\rho(x_0)}|\Psi_0(x)|^t\, dx 
&\leq \int_{B_1(0)}|x|^{(2-N)t}\, dx +\int_{B_\rho(x_0)\backslash B_1(0)}|x|^{(2-N)t}\, dx\\
&\leq \frac{|S^{N-1}|}{(N-(N-2)t)}+|S^{N-1}| \rho^N
\end{align*}
and, if $N\geq 4$,
\begin{align}
\int_{B_\rho(x_0)}\left(|x|^{\frac{1-N}{2}}\right)^t\, dx \nonumber
&\leq \int_{B_1(0)}|x|^{\frac{t(1-N)}{2}}\, dx +\int_{B_\rho(x_0)\backslash B_1(0)}|x|^{\frac{t(1-N)}{2}}\, dx\\
&\leq \frac{|S^{N-1}|}{(N-t\frac{(N-1)}{2})}+|S^{N-1}| \rho^N.\label{eq:new-weak-ineq}
\end{align}
Combining these estimates with \eqref{eq:combined-concl}, we infer that $\|\Psi_\lambda\|_{L^{t}(B_{r+r_0}(x_0-y_0))}$ is bounded above by a constant $D=D(N,p,q,r_0,r,{\Lambda^*})>0$, and thus \eqref{eq:r-x-0-est-1} yields the claim.  In case $p=2^\ast$ we find, using the weak Young inequality,
\begin{equation}
  \label{eq:r-x-0-est-2}
\|u\|_{L^{2^\ast}(B_r(x_0))}=\|\Psi_\lambda\ast f\|_{L^{2^\ast}(B_r(x_0))}
\leq \|\Psi_\lambda\|_{L^{\frac{N}{N-2},w}(B_{r+r_0}(x_0-y_0))} \|f\|_{(2^\ast)'},
\end{equation}
where $L^{\frac{N}{N-2},w}(B_{r+r_0}(x_0-y_0))$ denotes the weak $L^{\frac{N}{N-2}}(B_{r+r_0}(x_0-y_0))$-space.  Since $\Psi_0\in L^{{\frac{N}{N-2}},w}(\R^N)$ and \eqref{eq:new-weak-ineq} also holds for $t= \frac{N}{N-2}$, we infer from \eqref{eq:combined-concl} that $\|\Psi_\lambda\|_{L^{\frac{N}{N-2},w}(B_{r+r_0}(x_0-y_0))}
$ is bounded above by a constant $D=D(N,p,q,r_0,r,{\Lambda^*})>0$, and thus \eqref{eq:r-x-0-est-2} yields the claim in this case.
\end{proof}

\section{Some bounds for arbitrary solutions of the nonlinear problem}
\label{sec:priori-bounds-solut}

For the remainder of the paper, we fix a function $Q\in L^\infty_c(\R^N)\backslash\{0\}$ with $Q\geq 0$ a.e. on $\R^N$. Moreover, we set 
\begin{equation}
  \label{eq:def-r-0}
{r_Q}:=\text{diam}(\text{supp }Q) \in (0,\infty).
\end{equation}
We note that for $f \in L^1_{loc}(\R^N)$ we then have $Q f \in L^1_c(\R^N)$, and therefore $\mR_\lambda(Qf)$ is given as the convolution $ \Psi_\lambda \ast (Q f)$ a.e. on $\R^N$.  In the following, we consider nonlinear fixed point problems of the form 
\begin{equation}\label{eqn:fp_u-gen}
u=\mR_\lambda\left(Q|u|^{p-2}u + \varphi \right)= \Psi_\lambda \ast \left(Q|u|^{p-2}u + \varphi \right), \qquad u \in L^p_{loc}(\R^N)
\end{equation}
with $p>2$ and a given function $\varphi \in L^\infty_c(\R^N)$. We point out that, if $2<p<2^*$ and $u$
 is a solution of \eqref{eqn:fp_u-gen}, then we have $Q|u|^{p-2}u + \varphi \in L^{p'}_c(\R^N) \subset L^{(2^\ast)'}(\R^N)$, and thus Lemma~\ref{lem:reg1bis} implies that $u \in W^{2,(2^\ast)'}_{\text{loc}}(\R^N)\cap L^{2^\ast}(\R^N)$ is 
a strong solution of 
$$
 -\Delta u -\lambda u=  Q(x)|u|^{p-2}u+ \varphi  \qquad \text{in $\R^N$.}
$$
In particular, in the case $\phi \equiv 0$, $u$ is a strong solution of \eqref{eq:eq-general-schr-1}. We first need a relative $L^\infty$ estimate for solutions of \eqref{eqn:fp_u-gen}.
\begin{proposition}\label{prop:l_infty}
Let $2<p<2^\ast$, ${\Lambda^*}$, $\Lambda_*>0$ and $\lambda \in [-\Lambda_*, {\Lambda^*}]$, $\varphi\in L^\infty_c(\R^N)$ be given.
Then every solution $u \in L^p_{loc}(\R^N)$ of  \eqref{eqn:fp_u-gen} is contained in $L^\infty(\R^N) \cap W^{2,t}_{loc}(\R^N) \cap 
C^{1,\gamma}_{loc}(\R^N)$ for all $t \in [1,\infty), \gamma \in (0,1)$. Moreover, there exist constants 
$$
C=C\bigl(N,p,{r_Q},{\Lambda^*},\Lambda_*,\|Q\|_\infty\bigr)>0\qquad \text{and}\qquad m=m(N,p)\in\N
$$ 
independent of $u$, $\varphi$ and $\lambda$ such that
\begin{equation}
  \label{eq:infty-est}
\|u\|_\infty\leq C\left(\|Q|u|^{p-2}u\|_{p'}+\|Q|u|^{p-2}u\|_{p'}^{(p-1)^m}
+\|\varphi\|_\infty+\|\varphi\|_\infty^{(p-1)^m}\right).
\end{equation}
\end{proposition}

\begin{proof}
Let $f=Q|u|^{p-2}u$, so that $u$ solves the equation $u=\mR_\lambda(f+\varphi)$. 
Consider a strictly decreasing sequence 
$(R_k)_{k\in\N_0}$ of positive radii satisfying $R_0=2$ and
$R_k>1$ for all $k$.
Let $x_0\in\R^N$, $-\Lambda_*\leq\lambda\leq {\Lambda^*}$. 
By Lemma~\ref{lem:reg1bis}(b) with $t=(2^\ast)'$ and Lemma~\ref{lem:reg2},
we obtain $u\in W^{2,(2^\ast)'}_{\text{loc}}(\R^N)$ and constants $\tilde{C}_0=\tilde{C}_0(N,{\Lambda^*}, \Lambda_*)>0$,
$D=D\bigl(N,p,{r_Q},{\Lambda^*}\bigr)>0$ such that
\begin{align*}
\|u\|_{W^{2,(2^\ast)'}(B_{R_1}(x_0))}&\leq \tilde{C}_0 \left( \|u\|_{L^{(2^\ast)'}(B_2(x_0))}
+\|f+\varphi\|_{L^{(2^\ast)'}(B_2(x_0))}\right)\\
&\leq \tilde{C}_0\left(D\|f+\varphi\|_{p'}+\|f+\varphi\|_{L^{(2^\ast)'}(B_2(x_0))}\right)\\
&\leq C_0\left(\|f\|_{p'}+\|\varphi\|_\infty\right),
\end{align*}
where $C_0=C_0\bigl(N,p,{r_Q},{\Lambda^*},\Lambda_*\bigr)$.

From Sobolev's embedding theorem, there is for every $1\leq q\leq 2^\ast$ a constant 
$\kappa_q^{(0)}=\kappa_q^{(0)}(N,q)>0$ such that
$$
\|u\|_{L^q(B_{R_1}(x_0))}\leq \kappa_q^{(0)} C_0 \left(\|f\|_{p'}+\|\varphi\|_\infty\right).
$$
Consequently, with $t_1:=\frac{2^\ast}{p-1}$ there holds
$$
\|f\|_{L^{t_1}(B_{R_1}(x_0))}=\|Q|u|^{p-2}u\|_{L^{t_1}(B_{R_1}(x_0))}
\leq 2^{p-2}\|Q\|_\infty (\kappa_{2^\ast}^{(0)} C_0)^{p-1} \left(\|f\|_{p'}^{p-1}+\|\varphi\|_\infty^{p-1}\right).
$$
With Lemma~\ref{lem:reg1bis}(b)  it follows that $u\in W^{2,t_1}_{\text{loc}}(\R^N)$ and that
\begin{align*}
\|u\|_{W^{2,t_1}(B_{R_2}(x_0))}&\leq \tilde{C}_1 \left( \|u\|_{L^{t_1}(B_{R_1}(x_0))}+\|f+\varphi\|_{L^{t_1}(B_{R_1}(x_0))}\right)\\
&\leq \tilde{C}_1\Bigl[\kappa_{t_1}^{(0)}C_0\left(\|f\|_{p'}+\|\varphi\|_\infty\right)+\|\varphi\|_{L^{t_1}(B_{R_1}(x_0))}\\
&\qquad+2^{p-2}\|Q\|_\infty(\kappa_{2^\ast}^{(0)}C_0)^{p-1}\left(\|f\|_{p'}^{p-1}+\|\varphi\|_\infty^{p-1}\right)\Bigr]\\
&\leq C_1\left( \|f\|_{p'}+\|f\|_{p'}^{p-1}+\|\varphi\|_\infty+\|\varphi\|_\infty^{p-1}\right),
\end{align*}
where $C_1=C_1\bigl(N,p,{r_Q},{\Lambda^*}, \Lambda_*, \|Q\|_\infty\bigr)$.

If $t_1\geq \frac{N}{2}$, Sobolev's embedding theorem gives for each $1\leq q<\infty$ the existence of a constant 
$\kappa^{(1)}_q=\kappa_q^{(1)}(N,p,q)>0$ such that 
$$
\|u\|_{L^q(B_{R_2}(x_0))}\leq \kappa^{(1)}_q C_1\left(\|f\|_{p'}+\|f\|_{p'}^{p-1}+\|\varphi\|_\infty+\|\varphi\|_\infty^{p-1}\right).
$$ 
As a consequence, we obtain 
$$
\|f\|_{L^q(B_{R_2}(x_0))}\leq 4^{p-2}\|Q\|_\infty (\kappa_{q(p-1)}^{(1)} C_1)^{p-1} \left(\|f\|_{p'}^{p-1}+\|f\|_{p'}^{(p-1)^2}
+\|\varphi\|_\infty^{p-1}+\|\varphi\|_\infty^{(p-1)^2}\right)
$$
for all $1\leq q<\infty$. From Lemma~\ref{lem:reg1bis}(b) we obtain 
$u\in W^{2,N}_{\text{loc}}(\R^N)$ and since $R_2>1$,
\begin{align*}
\|u\|_{W^{2,N}(B_1(x_0))}&\leq \tilde{C}_2 \left( \|u\|_{L^N(B_{R_2}(x_0))}+\|f+\varphi\|_{L^N(B_{R_2}(x_0))}\right)\\
&\leq \tilde{C}_2\Bigl\{\kappa_N^{(1)}C_1\left(\|f\|_{p'}+\|f\|_{p'}^{p-1}+\|\varphi\|_\infty+\|\varphi\|_\infty^{p-1}\right)
+\|\varphi\|_{L^N(B_{R_2}(x_0))}\\
&\;+4^{p-2}\|Q\|_\infty(\kappa_{N(p-1)}^{(1)}C_1)^{p-1}
\left(\|f\|_{p'}^{p-1}+\|f\|_{p'}^{(p-1)^2}+\|\varphi\|_\infty^{p-1}+\|\varphi\|_\infty^{(p-1)^2}\right)\Bigr\}\\
&\leq C_2\left( \|f\|_{p'}+\|f\|_{p'}^{(p-1)^2}+\|\varphi\|_\infty+\|\varphi\|_\infty^{(p-1)^2}\right),
\end{align*}
where $C_2=C_2\bigl(N,p,{r_Q},{\Lambda^*}, \Lambda_*, \|Q\|_\infty\bigr)$.
By Sobolev's embedding theorem, there is a constant $\kappa_\infty=\kappa_\infty(N)>0$ for which
$$
\|u\|_{L^\infty(B_1(x_0))}\leq \kappa_\infty C_2\left( \|f\|_{p'}+\|f\|_{p'}^{(p-1)^2}+\|\varphi\|_\infty+\|\varphi\|_\infty^{(p-1)^2}\right).
$$
Hence the conclusion holds in this case with $C=\kappa_\infty C_2$ and $m=2$.

If $t_1<\frac{N}{2}$, we infer from Sobolev's embedding theorem that 
$$
\|u\|_{L^q(B_{R_2}(x_0))}\leq \kappa^{(1)}_q C_1\left(\|f\|_{p'}+\|f\|_{p'}^{p-1}+\|\varphi\|_\infty+\|\varphi\|_\infty^{p-1}\right)
$$ 
for each $1\leq q\leq \frac{Nt_1}{N-2t_1}$, where $\kappa_q^{(1)}=\kappa_q^{(1)}(N,p,q)$.
Therefore, setting $t_2:=\frac{Nt_1}{(N-2t_1)(p-1)}$, we obtain
$$
\|f\|_{L^{t_2}(B_{R_2}(x_0))}\leq 4^{p-2}\|Q\|_\infty (\kappa_{t_2(p-1)}^{(1)} C_1)^{p-1} 
\left(\|f\|_{p'}^{p-1}+\|f\|_{p'}^{(p-1)^2}+\|\varphi\|_\infty^{p-1}+\|\varphi\|_\infty^{(p-1)^2}\right).
$$
Using again Lemma~\ref{lem:reg1bis}(b), we find $u\in W^{2,t_2}_{\text{loc}}(\R^N)$
and
\begin{align*}
\|u\|_{W^{2,t_2}(B_{R_3}(x_0))}&\leq \tilde{C}_2 \left( \|u\|_{L^{t_2}(B_{R_2}(x_0))}+\|f+\varphi\|_{L^{t_2}(B_{R_2}(x_0))}\right)\\
&\leq \tilde{C}_2\Bigl\{\kappa_{t_2}^{(1)}C_1\left(\|f\|_{p'}+\|f\|_{p'}^{p-1}+\|\varphi\|_\infty+\|\varphi\|_\infty^{p-1}\right)
+\|\varphi\|_{L^{t_2}(B_{R_2}(x_0))}\\
&\;\; +4^{p-2}\|Q\|_\infty(\kappa_{t_2(p-1)}^{(1)}C_1)^{p-1}\left(\|f\|_{p'}^{p-1}+\|f\|_{p'}^{(p-1)^2}\!\!
+\|\varphi\|_\infty^{p-1}+\|\varphi\|_\infty^{(p-1)^2}\right)\Bigr\}\\
&\leq C_2\left( \|f\|_{p'}+\|f\|_{p'}^{(p-1)^2}+\|\varphi\|_\infty+\|\varphi\|_\infty^{(p-1)^2}\right),
\end{align*}
where $C_2=C_2\bigl(N,p,{r_Q}, {\Lambda^*}, \Lambda_*, \|Q\|_\infty\bigr)$.

Remarking that $t_2>t_1$, since $p<2^\ast$, we may iterate the procedure. At each step we find some constant 
$C_k=C_k\bigl(N,p,{r_Q}, {\Lambda^*}, \Lambda_*, \|Q\|_\infty\bigr)$ such that the estimate 
$$
\|u\|_{W^{2,t_k}(B_{R_{k+1}}(x_0))}\leq C_k\left(\|f\|_{p'}+\|f\|_{p'}^{(p-1)^k}
+\|\varphi\|_\infty+\|\varphi\|_\infty^{(p-1)^k}\right)
$$
holds and where $t_k$ is defined recursively via $t_0=(2^\ast)'$ and 
$t_{k+1}=\frac{Nt_k}{(N-2t_k)(p-1)}$, as long as $t_k<\frac{N}{2}$.
Since $t_{k+1}\geq \frac{t_1}{p'}\,t_k$ and since $t_1>p'$, we reach after finitely many steps 
$t_\ell\geq\frac{N}{2}$, where 
$\ell$ only depends on $N$ and $p$.
Since $R_k>1$ for all $k$, applying Lemma~\ref{lem:reg1bis}(b) one more time 
and arguing as above, we obtain $u\in W^{2,N}_{\text{loc}}(\R^N)$ as well as the estimate
$$
\|u\|_{W^{2,N}(B_1(x_0))}\leq C_{\ell+1}\left(\|f\|_{p'}+\|f\|_{p'}^{(p-1)^{\ell+1}}
+\|\varphi\|_\infty+\|\varphi\|_\infty^{(p-1)^{\ell+1}}\right),
$$
where $C_{\ell+1}=C_{\ell+1}\bigl(N,p,{r_Q},{\Lambda^*},\Lambda_*,\|Q\|_\infty\bigr)$.
Then, Sobolev's embedding theorem gives a constant $\kappa_\infty=\kappa_\infty(N)$ for which
$$
\|u\|_{L^\infty(B_1(x_0))}\leq \kappa_\infty C_{\ell+1}\left(\|f\|_{p'}+\|f\|_{p'}^{(p-1)^{\ell+1}}
+\|\varphi\|_\infty+\|\varphi\|_\infty^{(p-1)^{\ell+1}}\right)
$$
holds. Setting $m=\ell+1$ concludes the proof of \eqref{eq:infty-est}. 
Applying Lemma~\ref{lem:reg1bis}(b) once more, we see that $u \in W^{2,t}_{loc}(\R^N)$ for all $t <\infty$, and thus 
$u \in C^{1,\gamma}_{loc}(\R^N)$ for all $\gamma \in (0,1)$.
\end{proof}

In the following, we let 
$$
\Omega= \{x \in \R^N \::\: Q(x)>0\}
$$
Our next aim is to obtain a lower bound on the $L^\infty(\Omega)$-norm of nontrivial solutions of \eqref{eqn:fp_u-gen} with $\varphi \equiv 0$. 
 We have the following result: 

 \begin{lemma}
\label{sec:infinity-norm-lower-bounds}
Let ${\Lambda^*}>0$, $2< p \le 2^*$, and let $Q$ and $\Omega$ be as above. Then there exists a constant $\delta_0=\delta_0\bigl(N,p,{\Lambda^*}, Q\bigr)>0$ such that
$$
 \mR_\lambda (Q |u|^{p-2}u) \not = \tau u 
$$
for any $\lambda \le {\Lambda^*}$, $\tau \ge 1$ and $u \in L^\infty_{loc}(\R^N)$ with $0<\|u\|_{L^\infty(\Omega)}\le \delta_0$.
 \end{lemma}

 \begin{proof}
Let $r>0$ be such that $\Omega \subset B_r(0)$, and let $D>0$ be such that Lemma~\ref{lem:reg2} holds with the given values of $r, {\Lambda^*}, p$, $q=p$ and $r_0:={r_Q}$ defined by \eqref{eq:def-r-0}.  Moreover, put 
$$
\delta_0 := \frac{1}{2}\bigl(D \|Q\|_\infty\bigr)^{-\frac{1}{p-2}}  |\Omega|^{-\frac{1}{p}},
$$ 
and let $u \in L^\infty_{loc}(\R^N)$ with $\|u\|_{L^\infty(\Omega)}>0$ and $\tau \ge 1$ be such that  $ \mR_\lambda (Q |u|^{p-2}u)  = \tau u $. Then Lemma~\ref{lem:reg2} implies that 
 \begin{align*}
\|u\|_{L^p(\Omega)} \le \|\mR_\lambda (Q |u|^{p-2}u)\|_{L^p(\Omega)} \le \|\mR_\lambda (Q |u|^{p-2}u)\|_{L^p(B_r(0))} &\le D \|Q |u|^{p-2}u\|_{L^{p'}(\R^N)}\\
&\le D\|Q\|_\infty \|u\|_{L^p(\Omega)}^{p-1}
 \end{align*}
and hence, 
$$
1 \le D\|Q\|_\infty \|u \|_{L^p(\Omega)}^{p-2} \le D\|Q\|_\infty |\Omega |^{\frac{p-2}{p}} \|u \|_{L^\infty (\Omega)}^{p-2}.
$$
We thus conclude that $\|u\|_{L^\infty(\Omega)} \ge \bigl(D \|Q\|_\infty\bigr)^{-\frac{1}{p-2}}  |\Omega|^{-\frac{1}{p}} = 2 \delta_0>\delta_0$. This shows the claim.
 \end{proof}

\section{A priori bounds for nonnegative solutions}
\label{sec:priori-bounds-nonn}

As before, we assume $Q\in L^\infty_c(\R^N)$ is nonnegative, $Q\not\equiv 0$, let $\Omega=\{x\in\R^N\, :\, Q(x)>0\}$ and ${r_Q}=\text{diam}(\Omega)$. 
In this section we wish to derive $L^\infty$-bounds on nonnegative solutions of \eqref{eqn:fp_u-gen}. 
For this we put 
$$
\lambda_Q:= \left(\frac{y_{\frac{N-2}{2}}^{(1)}}{{r_Q}}\right)^2,
$$
and we focus our attention on the range  
\begin{equation}\label{eqn:lambda0}
\lambda\in(-\infty,\lambda_Q).
\end{equation} 
By Lemma~\ref{lem:bessel}(iii), it follows from \eqref{eqn:lambda0} that 
\begin{equation}
  \label{eq:pos-psi-lambda}
\Psi_\lambda(x-y)>0 \qquad \text{for all $x,y\in \overline \Omega$.}  
\end{equation}
As a consequence of \eqref{eq:pos-psi-lambda}, we immediately obtain the following positivity property for solutions of \eqref{eqn:fp_u-gen}.

\begin{lemma}\label{lem:positive}
Let $Q$, ${r_Q}$ and $\lambda_Q$ be as above, $\lambda < \lambda_Q$ and let $\varphi \in L^\infty_c(\R^N)$ be nonnegative. Then every nonnegative nontrivial solution $u \in L^p_{loc}(\R^N)$ of \eqref{eqn:fp_u-gen} is strictly positive in $\overline \Omega$.
\end{lemma}

In the following, it will be useful to reformulate problem~\eqref{eqn:fp_u-gen}. For this we consider, for $\lambda \in \R$, the linear operator
$\mK_\lambda:  C(\oO) \to C(\oO)$ defined by 
\begin{equation}
  \label{eq:def-K-lambda}
\mK_\lambda(f)=[\Psi_\lambda \ast (Qf)]\Big|_{\oO}\qquad \text{for $f \in C(\oO)$.}
\end{equation}
Here and in the following, the convolution with $Qf$ is understood by considering the trivial extension of $Qf$ on $\R^N \setminus \Omega$, i.e.,  
\begin{equation}
  \label{eq:extension-formula}
[\Psi_\lambda \ast (Qf)](x) = \int_{\Omega} \Psi_\lambda(x-y) Q(y)f(y)\,dy \qquad \text{for $x \in \R^N$.}
\end{equation}
Since $Qf \in L^\infty(\Omega)$ and the singularity of $\Psi_\lambda$ is of the order 
$\Psi_\lambda(z)= O(|z|^{2-N})$ as $z \to 0$, it is easy to see that $\Psi_\lambda \ast (Qf)$ defines a continuous function for $f \in C(\oO)$. Hence $\mK_\lambda: C(\oO) \to C(\oO)$ is well defined. 
Moreover, there is a one-to-one correspondence between solutions of 
\begin{equation}\label{eqn:fp_u}
u=\Psi_\lambda \ast \left(Q|u|^{p-2}u\right), \qquad u \in L^p_{loc}(\R^N)
\end{equation}
and solutions of the operator equation
\begin{equation}
  \label{eq:K-equation}
u = \mK_\lambda(|u|^{p-2}u), \qquad u \in C(\oO).
\end{equation}
Indeed, if $u \in L^p_{loc}(\R^N)$ solves \eqref{eqn:fp_u}, then $u$ is continuous by Proposition~\ref{prop:l_infty}, and the restriction of $u$ to $\Omega$ satisfies \eqref{eq:K-equation}. On the other hand, if $u \in C(\oO)$ solves \eqref{eq:K-equation},  we may extend $u$ on $\R^N$ by \eqref{eq:extension-formula} with $f= |u|^{p-2}u$, and then $u$ satisfies \eqref{eqn:fp_u}. We need following properties of the operators $\mK_\lambda$, $\lambda \in \R$.

\begin{lemma}
\label{comp-strictly-positive}
For every $\lambda \in \R$, we have: 
\begin{itemize}
\item[(a)] $\mK_\lambda$ is a compact linear operator.  
\item[(b)] If $\lambda < \lambda_Q$, then $\mK_\lambda$ is strictly positivity preserving, i.e., we have 
$\min \limits_{\oO} \mK_\lambda(f)>0$ for every nonnegative nontrivial function $f \in C(\overline \Omega)$. 
\end{itemize}
\end{lemma}

\begin{proof}
(a) Choosing $R>0$ such that $\oO\subset B_R(0)$, 
a standard bootstrap argument (using  Lemma~\ref{lem:reg1bis}  or
\cite[Proposition A.1]{evequoz-weth-dual}) provides the estimate
$$
\|\Psi_\lambda\ast\left(Q f\right)\|_{W^{2,N}(B_R(0))}\leq C \|Q f\|_\infty, \quad\text{for all }f \in C(\oO),
$$
where $C$ does not depend on $f$.
By the Rellich-Kondrachov theorem, this implies the compactness of the operator 
$f\mapsto \Psi_\lambda\ast\left(Q  f\right)$ between the spaces $C(\oO)$ and $C(\overline{B_R(0)})$,
and from this the compactness of $\mK_\lambda$ follows.\\
(b) If $\lambda < \lambda_Q$ and $f\in C(\oO)$ is nonnegative with $f\not\equiv 0$, then \eqref{eq:pos-psi-lambda} implies
that $\Psi_\lambda \ast(Qf)>0$ in $\oO$. Since $\oO$ is compact, it follows that 
$\min \limits_{\oO} \mK_\lambda(f)= \min \limits_{\oO} \Psi_\lambda \ast(Qf) >0$.
\end{proof}

We now state our first result on $L^\infty$-bounds on nonnegative solutions of \eqref{eqn:fp_u-gen}.  It is based on assumption $(A1)$ from the introduction and therefore requires a restriction of the range of admissible exponents $p$.

\begin{proposition}\label{prop:a_priori_w_restr}
Let $Q$, $\Omega$ be as above, let $2<p<\frac{2^\ast}{2}+1$, and let $\Lambda_*,\Lambda^*>0$ be such that ${\Lambda^*}< \lambda_Q$. Then there exists a constant $C>0$ such that 
\begin{equation}\label{eqn:uniform_sup_w}
\|u\|_{L^\infty(\overline \Omega)}+t\leq C
\end{equation} 
for all solutions $(t,u)\in[0,\infty)\times C(\oO)$, $u \ge 0$ 
of the problem
\begin{equation}\label{eqn:fp_w_modif}
u=\mK_\lambda(u^{p-1}+t)
\end{equation}
with $\lambda\in[-\Lambda_*,{\Lambda^*}]$.
\end{proposition}

\begin{proof}
We first observe that, as a consequence of Lemma~\ref{comp-strictly-positive} and the Krein-Rutman Theorem (see e.g. \cite[Theorem 7.C]{zeidler}), 
the operator $\mK_0$ has a strictly positive eigenfunction $f_1\in C(\oO)$ associated to the 
positive eigenvalue $\nu_1=r(\mK_0)>0$, i.e., we have 
$$
\mK_0 f_1 = \Psi_0 \ast (Qf_1)= \nu_1 f_1. 
$$ 
Here 
$r(\mK_0)$ denotes the spectral radius of $\mK_0$. Let $\lambda\in[-\Lambda_*,{\Lambda^*}]$ and consider $t\geq 0$ and $ u\in C(\oO)$ nonnegative
such that \eqref{eqn:fp_w_modif} holds. We first note that in case $u=0$ we have $0=\mK_\lambda(t)= t \mK_\lambda(1)$, and this yields $t=0$ by Lemma~\ref{comp-strictly-positive}(b). We may thus assume that $u \not= 0$ from now on, and then Lemma~\ref{comp-strictly-positive}(b) implies that $u$ is strictly positive in $\oO$.
Moreover, the function 
$$
u_0:= \mK_0(u^{p-1}+t)= \Psi_0 * (Qu^{p-1}+tQ)
$$ 
is also strictly positive in $\overline \Omega$, and by Lemma~\ref{lem:bessel} (ii)--(iv) we have 
$$
0< \eps _0 \Psi_0(x-y) \le \Psi_\lambda(x-y) \le \gamma_N \Psi_0(x-y) \qquad \text{for $x,y\in \oO$, $\lambda\in[-\Lambda_*,{\Lambda^*}]$}
$$ 
and therefore  
\begin{equation}\label{eqn:ineq_u0}
\eps_0 u_0\leq u\leq \gamma_N u_0\quad\text{ in }\oO, 
\end{equation}
where $\gamma_N$ is given in \eqref{eq:def-gamma-N} and  $\eps_0=\eps_0(\Lambda_*,\Lambda^*,r_Q)>0$. We thus find that 
\begin{align*}
\nu_1\!  \int_{\Omega}(Qu^{p-1}&+tQ) f_1\, dx =  \int_{\Omega}(Qu^{p-1}+tQ)[\Psi_0 * (Q f_1)]\, dx=
\int_{\Omega}[\Psi_0 * (Qu^{p-1}+tQ)] (Q f_1)\, dx\\
&=  \int_{\Omega}u_0 (Q f_1)\, dx
  \leq \frac1{\eps_0}\int_{\Omega}Qu f_1\, dx \leq \frac1{\eps_0}\left(\int_{\Omega}Qu^{p-1}f_1\, dx\right)^{\frac1{p-1}}
\left(\int_{\Omega}Qf_1\, dx\right)^{\frac{p-2}{p-1}},
\end{align*}
using H\"older's inequality. As a consequence, 
$$
\int_{\Omega}Qu^{p-1}f_1\, dx\leq \frac1{(\eps_0\nu_1)^{\frac{p-1}{p-2}}}\int_{\Omega}Qf_1\, dx,
$$
and since $\min \limits_{\oO} f_1 >0$, we obtain the estimates
\begin{align}
\int_{\Omega}Qu^{p-1}\, dx&\leq \frac1{\min\limits_{\overline{\Omega}}f_1\ (\eps_0\nu_1)^{\frac{p-1}{p-2}}}
\int_{\Omega}Qf_1\, dx=:\kappa_1, \label{eqn:u_p-1}\\
t&\leq \frac1{(\eps_0\nu_1)^{\frac{p-1}{p-2}}}=:\kappa_2,\label{eqn:t}
\end{align}
noting that $\kappa_1, \kappa_2$ are independent of $\lambda\in[-\Lambda_*,{\Lambda^*}]$. In particular, setting $v:= Q u^{p-1}+ tQ$, we find that 
\begin{equation}
  \label{eq:l-1-est-v}
 \|v\|_{L^1(\Omega)} \le \kappa_1 + \kappa_2 \|Q\|_{L^1(\Omega)}, 
\end{equation}
and that 
\begin{equation}
  \label{eq:l-p-prime-est-v}
 \|v\|_{L^{p'}(\Omega)} \le \|Q\|_{L^\infty(\Omega)}^{\frac{1}{p}} \Bigl(\int_{\Omega}Qu^p\, dx\Bigr)^{\frac{1}{p'}}   + \kappa_2 \|Q\|_{L^{p'}(\Omega)}.
\end{equation}
Using \eqref{eqn:ineq_u0}, \eqref{eq:l-1-est-v} and \eqref{eq:l-p-prime-est-v} together with the fact that $u_0:= \Psi_0 * v$ in $\Omega$, we find that 
\begin{align*}
&\int_{\Omega}Qu^p\, dx \le \int_{\Omega}Qu^p\, dx + t\int_{\Omega}Qu\, dx \leq \gamma_N \int_{\Omega}v u_0\, dx= \gamma_N \int_{\Omega} v\, (\Psi_0 * v)\, dx\\
&\qquad \le \gamma_N \|\Psi_0\|_{\frac{N}{N-2},w} \|v\|_{L^{(2^*)'}(\Omega)}^2 \le \gamma_N \|\Psi_0\|_{\frac{N}{N-2},w} \|v\|_{L^{1}(\Omega)}^{2-2\alpha}  \|v\|_{L^{p'}(\Omega)}^{2\alpha} \\
&\qquad\le \gamma_N \|\Psi_0\|_{\frac{N}{N-2},w} 
\Bigl(\kappa_1 + \kappa_2 \|Q\|_{L^1(\Omega)}\Bigr)^{2-2\alpha} 
\Bigl( \|Q\|_{L^\infty(\Omega)}^{\frac{1}{p}} \Bigl(\int_{\Omega}Qu^p\, dx\Bigr)^{\frac{1}{p'}}   + \kappa_2 \|Q\|_{L^{p'}(\Omega)}\Bigr)^{2\alpha}. 
\end{align*}
Here, as in the proof of Lemma~\ref{lem:reg1bis-0},  $\|\Psi_0\|_{\frac{N}{N-2},w}$ denotes the weak $L^{\frac{N}{N-2},w}(\R^N)$-norm of $\Psi_0$, and the weak Young inequality has been used. Moreover, the exponent $\alpha$ arising from interpolation via H{\"o}lder's inequality is given by $\alpha = \frac{1-\frac{1}{(2^*)'}}{1-\frac{1}{p'}}$.   Since by assumption $p<\frac{2^\ast}2+1$, it follows that $\frac{2\alpha}{p'} \in (0,1)$. It thus follows from the latter estimate that there exists a constant $\kappa_3>0$, depending only on $N$, $Q$ and $p$ but not on $\lambda\in[-\Lambda_*,{\Lambda^*}]$ and $u$, such that 
\begin{equation}\label{eqn:u_p}
\int_{\Omega}Qu^p\, dx\leq \kappa_3.
\end{equation}
Extending $u$ to $\R^N$ by \eqref{eq:extension-formula} with $f=u^{p-1}+t$ and using Proposition~\ref{prop:l_infty}, we find that
$$
\|u\|_{L^\infty(\R^N)} \leq C\left(\|Q\|_\infty^p\kappa_3^{\frac1{p'}}+\|Q\|_\infty^{p(p-1)^m}\kappa_3^{\frac{(p-1)^m}{p'}} 
+\|Q\|_\infty \kappa_2 + \|Q\|_\infty^{(p-1)^m}\kappa_2^{(p-1)^m}\right)=:\kappa_4
$$
where the constant on the right-hand side is independent of $u$ and $\lambda\in[-\Lambda_*,{\Lambda^*}]$.
Together with  \eqref{eqn:t} this gives the uniform bound
\begin{equation}\label{eqn:w_a_priori}
\|u\|_{L^\infty(\overline \Omega)} +t\leq \kappa_4 +\kappa_2=:C
\end{equation}
for all nonnegative solutions $(t,u)$ of \eqref{eqn:fp_w_modif} with $\lambda\in[-\Lambda_*,{\Lambda^*}]$.
This concludes the proof.
\end{proof}

Next, we wish to derive a further result on $L^\infty$-bounds for nonnegative solutions of \eqref{eq:K-equation} which is based on assumption $(A2)$ from the introduction and therefore applies to all subcritical exponents $p$.  This result is inspired by work of Amann, L\'opez-G\'omez~\cite{amann-lopez98} and Berestycki, Capuzzo-Dolcetta and 
Nirenberg~\cite{berestycki-cd-nirenberg94} for indefinite semilinear elliptic problems on bounded domains. 

\begin{proposition}\label{prop:a_priori_w_all_p}
Let $Q$, $\Omega$ be as above, let $2<p<2^\ast$, and let $\Lambda_*,\Lambda^*>0$ be such that ${\Lambda^*}< \lambda_Q$. Assume furthermore that $\partial \Omega$ is of class $C^1$, and assume that there exists $\alpha>0$ and a positive function $\beta \in  C(\R^N)$ such that
\begin{equation}\label{eqn:Q_dist}
Q(x)= \beta(x) \text{dist}(x, \R^N \! \setminus \!\Omega)^\alpha \qquad \text{ for all $x\in \R^N$.}
\end{equation}
Then there exists a constant $C>0$ such that 
\begin{equation}
\label{eqn:uniform_sup_w_unrestr}
\|u\|_{L^\infty(\overline \Omega)}+t\leq C
\end{equation} 
for all solutions $(t,u)\in[0,\infty)\times C(\oO)$ of the problem~\eqref{eqn:fp_w_modif} with
$u \geq 0$ in $\Omega$.
\end{proposition}

\begin{proof}
We start by remarking that the first part of the proof of Proposition~\ref{prop:a_priori_w_restr}
works for all $2<p<2^\ast$. Hence, there is a constant $\kappa_2>0$ such that \eqref{eqn:t}
holds for all $t\geq 0$ which solve \eqref{eqn:fp_w_modif} for some $ u\geq 0$.

Suppose by contradiction that there exists some sequence $((\lambda_n,t_n,u_n))_n$ in 
$[-\Lambda_*,{\Lambda^*}]\times [0,\kappa_2] \times C(\oO)$ such that $u_n\geq 0$ in $\Omega$,
$$
u_n=\mK_{\lambda_n}(u_n^{p-1}+t_n) \qquad \text{in $C(\oO)$ for all $n \in \N$}
$$
and  
$$
M_n:= \|u_n\|_{L^\infty(\overline \Omega)}\to\infty \qquad \text{as $n\to\infty$.}
$$
For $n \in \N$, we extend $u_n$ canonically to all of $\R^N$ by \eqref{eq:extension-formula} 
with $\lambda= \lambda_n$ and $f= u_n^{p-1}+t_n$. Then Lemma~\ref{lem:reg1bis}(a) implies  
that $u_n \in L^{2^\ast}(\R^N)\cap W^{2,(2^\ast)'}_{\text{loc}}(\R^N)$ is a strong solution of 
$$
-\Delta u_n -\lambda_n u_n  = Qu_n^{p-1}+t_n Q \qquad \text{in $\R^N$.}
$$ 
Next, we note that by the assumption $\Lambda^* < \lambda_Q$ we may fix $\eps>0$ such that  
\begin{equation*}
\sqrt{{\Lambda^*}}({r_Q}+\eps)<y_{\frac{N-1}2}^{(1)}.
\end{equation*} 
We then put 
$$
\Omega_\eps:=\{x\in\R^N\, :\, \text{dist}(x,\Omega)<\eps\}, 
$$
and we claim that there exists a constant $\kappa_\eps>0$ such that 
\begin{equation}
  \label{eq:rel-kappa-est}
0 \le  u_n \leq \kappa_\eps M_n \qquad \text{in $\Omega_\eps\quad$ for all $n \in \N$.}
\end{equation}
Indeed, we recall that, by Lemma~\ref{lem:bessel} (ii) -- (iv), there exist   
constants $\delta_\eps>0$ and $\gamma>0$ such that
$$
\delta_\eps\Psi_0(z)\leq \Psi_{\lambda_n}(z)\leq \gamma \Psi_0(z)\quad\text{ for all $z \in \R^N$ with $|z|\leq {r_Q}+\eps$ and $n \in \N$.}
$$
Consequently, putting 
$$
\tilde u_n := \Psi_0 * (Q u_n^{p-1} + t_n Q) \;\in \; L^{2^*}(\R^N)\qquad \text{for $n \in \N$,}
$$
we deduce that 
\begin{equation}\label{eqn:bds_u_u0}
\delta_\eps \tilde u_n(x) \leq u_n(x)\leq \gamma \tilde u_n(x)\quad\text{ for all }x\in\Omega_\eps.
\end{equation}
Moreover, we have 
$$
-\Delta \tilde u_n=Q u_n^{p-1}+t_nQ=0\quad\text{ in }\R^N\backslash\oO,
$$
and thus the maximum principle for harmonic functions shows that $\sup\limits_{\R^N}\tilde u_n = \max\limits_{\oO} \tilde u_n$. The bounds \eqref{eqn:bds_u_u0} therefore yield \eqref{eq:rel-kappa-est} with $\kappa_\eps= \frac{\gamma}{\delta_\eps}$, as claimed.

Next, let $x_n\in\oO$ be such that $u_n(x_n)=M_n$. 
Since $\oO$ is compact, we may assume, passing to a subsequence, that 
$$
x_n\to x^\ast \qquad \text{as $n\to\infty\quad$ for some $x^\ast\in \oO$.}
$$
Inspired by \cite{gidas.spruck81,amann-lopez98},  we now perform a rescaling of $u_n$.  For this we pass to a subsequence such that one of the following cases occurs.\\
{\bf Case 1:} $x^* \in \Omega$.\\
{\bf Case 2:} $x^* \in \partial \Omega$, and 
$M_n^{\frac{p-2}{\alpha+2}}\dist(x_n,\R^N \! \setminus \!\Omega) \to c \in [0,\infty)$.\\
{\bf Case 3:} $x^* \in \partial \Omega$, and 
$M_n^{\frac{p-2}{\alpha+2}}\dist(x_n, \R^N \! \setminus \! \Omega)  \to \infty$ as $n \to \infty$.\\
We then consider the rescaled functions  
$$
v_n \in L^{2^\ast}(\R^N)\cap W^{2,(2^\ast)'}_{\text{loc}}(\R^N),\qquad  v_n(y)=M_n^{-1}u_n(x_n+c_n y)
$$
with 
$$
c_n= \left \{
  \begin{aligned}
  &M_n^{\frac{2-p}{2}} &&\qquad \text{in Case 1;}\\
  &M_n^{\frac{2-p}{\alpha+2}} &&\qquad \text{in Case 2;}\\
  &M_n^{\frac{2-p}{2}} \dist(x_n, \R^N \! \setminus \! \Omega)^{-\frac{\alpha}{2}} &&\qquad \text{in Case 3.} 
  \end{aligned}
\right.
$$ 
We note that in all cases we have $\lim \limits_{n \to \infty} c_n = 0.$ Moreover, the functions $v_n$ are strong solutions of the equations 
\begin{equation}
\label{eq:v-n-rescaled}
-\Delta v_n - \lambda_n c_n^2 v_n =  Q_n v_n^{p-1} + 
t_n M_n^{1-p} Q_n\qquad \text{in $\R^N$}
\end{equation}
with 
\begin{equation}
  \label{eq:definition-Q-n}
Q_n \in L^\infty_c(\R^N), \qquad Q_n(y)=M_n^{p-2} c_n^2 Q(x_n + c_n y).
\end{equation}
We also put
$$
\Omega_n:= \{y \in \R^N\::\: x_n + c_n y \in \Omega\}
$$
and 
$$
\Omega_{n,\eps}:= \{y \in \R^N\::\: x_n + c_n  y \in \Omega_\eps\}. 
$$
We note that, by \eqref{eq:rel-kappa-est} we have 
\begin{equation}
  \label{eq:local-bound-v-n}
0 \le v_n \le \kappa_\eps \qquad \text{in $\Omega_{n,\eps}$}\qquad \text{and}\qquad v_n(0)=1.
\end{equation}
By construction and since $x^*$ is contained in the interior of $\Omega_\eps$, we see that the domains $\Omega_{n,\eps}$ converge to $\R^N$ in the sense that for every $R>0$ there exists $n_R$ such that $B_R(0) \subset \Omega_{n,\eps}$ for $n \ge n_R$. Hence we infer from \eqref{eq:local-bound-v-n} that the sequence $(v_n)_n$ is locally uniformly bounded on $\R^N$. We also note that 
\begin{equation}
 \label{eq:rescaled-boundary-distance}
\dist(x_n+c_n y, \R^N \! \setminus \! \Omega) = c_n \dist(y, \R^N \! \setminus \!\Omega_n)  \qquad \text{for every $y \in \R^N$}
 \end{equation}
and therefore, by \eqref{eqn:Q_dist} and \eqref{eq:definition-Q-n},  
\begin{equation}
  \label{eq:definition-Q-n-1}
Q_n(y)= \beta (x_n + c_n y)    M_n^{p-2} c_n^{2+\alpha}  \dist(y,\R^N \! \setminus \!\Omega_n)^\alpha \qquad \text{for $y \in \Omega_n$.}
\end{equation}
In the following we distinguish  Cases 1-3 above.\\
{\bf Case 1:} In this case we have
$$
Q_n(y)= Q(x_n+c_n y)\qquad \text{for $y \in \R^N$} 
$$
 by \eqref{eq:definition-Q-n} and the definition of $c_n$, $n \in \N$. Since $x_n \to x^* \in \Omega$ and $c_n \to 0$ as $n \to \infty$, we deduce that 
$$
Q_n \to Q(x^*)>0 \qquad \text{locally uniformly on $\R^N$ as $n \to \infty$.}
$$
Using interior $W^{2,q}$-estimates and the fact that 
$(\lambda_n)_n$ and $(t_n)_n$ are bounded sequences whereas $M_n \to \infty$ and $c_n \to 0$ as $n \to \infty$, we deduce from \eqref{eq:v-n-rescaled} that $v_n$ converges (up to a subsequence) 
locally uniformly to a bounded nonnegative strong solution $v$ of
$$
-\Delta v = Q(x^\ast)v^{p-1}\quad\text{ in } \R^N.
$$
As a consequence of Schauder estimates, $v$ is a classical solution of the above equation and,
since $Q(x^\ast)>0$, Theorem 1.2 in \cite{gidas.spruck81} implies that $v=0$, in contradiction to the fact that 
$v_n(0)=1$ for all $n$. Hence Case 1 does not occur.\\
{\bf Case 2:} Since $\Omega$ is of class $C^1$ and $x^* \in \partial \Omega$, there exists an affine half space $H$ such that $\partial  H$ is tangent to $\partial \Omega$ at $x^*$ and  
\begin{equation}
  \label{eq:dist-formula-1}
\dist(y, \R^N \! \setminus \!\Omega) = \dist(y, H)+o(|y-x^*|) \qquad \text{as $y \to x^*$.} 
\end{equation}
Consider the rescaled half space $H_n:= \{y \in \R^N\::\: x_n + c_n y \in H\}$. Then $\partial H_n$ is tangent to $\Omega_n$ at $z_n:= \frac{x^* - x_n}{c_n}$.  Since $c_n \to 0$ as $n \to \infty$, it follows from \eqref{eq:rescaled-boundary-distance} and \eqref{eq:dist-formula-1} that  
\begin{equation}
  \label{eq:dist-formula-2}
\dist(y, \R^N \! \setminus \!\Omega_n) = \dist(y, H_n) + o(1)  \qquad \text{locally uniformly in $y \in \R^N$ as $n \to \infty.$} 
\end{equation}
Furthermore, by definition of $c_n$ in Case 2, we have that $\dist(0, \R^N \! \setminus \! \Omega_n)= \frac{\dist(x_n,\R^N \! \setminus \!\Omega)}{c_n} \to c$ and thus also $\dist(0, H_n) \to c$ as $n \to \infty$ by \eqref{eq:dist-formula-2}. Passing to a subsequence, we may therefore assume that there exists an affine limit half space $H \subset \R^N$ such that 
$$
\lim_{n \to \infty} \dist(y, H_n) = \dist(y,H) \qquad  \text{locally uniformly in $y \in \R^N$ as $n \to \infty.$} 
$$
Combining this with \eqref{eq:dist-formula-2}, we find that 
$$
\lim_{n \to \infty} \dist(y,\R^N  \setminus \Omega_n)= \dist(y,H) 
\qquad  \text{locally uniformly in $y \in \R^N$ as $n \to \infty.$} 
$$
Consequently, by \eqref{eq:definition-Q-n-1} and the definition of $c_n$ in Case 2,  
$$
Q_n(y)= \beta (x_n + c_n y)  \dist(y,\R^N \! \setminus \! \Omega_n)^\alpha \to \beta(x^*) \dist(y,H)^\alpha
\qquad  \text{locally uniformly in $y \in \R^N$} 
$$
 as $n \to \infty.$ As in Case 1, we then find that $v_n$ converges (up to a subsequence) 
locally uniformly to a bounded nonnegative strong solution $v$ of
$$
-\Delta v = \beta(x^*) \dist(y,H)^\alpha v^{p-1},\qquad y \in \R^N.
$$
By \cite[Theorem B]{du.li:05} (see also \cite[Theorem 1.1]{polacik.quittner:05}), we then conclude that $v \equiv 0$, in contradiction to the fact that 
$v_n(0)=1$ for all $n$. Hence Case 2 does not occur.\\
{\bf Case 3:} In this case we find that  
$$
\dist(0,\R^N \! \setminus \! \Omega_n) =\frac{\dist(x_n,\R^N \! \setminus \!\Omega)}{c_n} = \Bigl(M_n^{\frac{p-2}{\alpha+2}} \dist(x_n,\R^N \! \setminus \!\Omega) \Bigr)^{\frac{2+\alpha}{2}} \to \infty,
$$
which implies that 
$$
 \dist(y,\R^N \! \setminus \!\Omega_n) = \dist(0,\R^N \! \setminus \!\Omega_n)\bigl(1 + o(1)\bigr) \qquad \text{locally uniformly in $y \in \R^N$ as $n \to \infty$.}
$$
Since, by definition of $c_n$, we also have 
$$
\dist(0,\R^N \! \setminus \!\Omega_n)= \frac{\dist(x_n,\R^N \! \setminus \!\Omega)}{c_n}= M_n^{\frac{2-p}{\alpha}} c_n^{-\frac{2+\alpha}{\alpha}}, 
$$
it thus follows that 
$$
\dist(y,\R^N \! \setminus \!\Omega_n)^\alpha = M_n^{2-p} c_n^{-2-\alpha}\bigl(1 + o(1)\bigr)\qquad \text{locally uniformly in $y \in \R^N$ as $n \to \infty$.}
$$
We thus find that 
$$
Q_n(y)= \beta (x_n + c_n y)    M_n^{p-2} c_n^{2+\alpha}  \dist(y,\R^N \! \setminus \!\Omega_n)^\alpha  \to \beta(x^*) \qquad \text{locally uniformly in $y \in \R^N$}
$$
as $n \to \infty$. Since $\beta(x^*)>0$, we arrive at a contradiction as in Case 1.\\
Since the above are the only possible cases, the proposition is proved.
\end{proof}

\section{A global branch of positive solutions}
\label{sec:glob-branch-posit}  
As before we assume that $Q\in L^\infty_c(\R^N)$ nonnegative, $Q\not\equiv 0$, and we let $\Omega \subset \R^N$ and ${r_Q}>0$, $\lambda_Q>0$ be given as in the previous sections.

In this section, we shall prove the existence of a connected component of positive solutions
of the problem \eqref{eq:K-equation}, where $\mK_\lambda$ is defined in \eqref{eq:def-K-lambda}. This will result from an application of the Leray-Schauder
continuation principle (see e.g.\cite{zeidler}). To set up the corresponding framework, we consider, for $p>2$, the nonlinear operator 
\begin{equation}
  \label{eq:F-operator-def}
F: \R\times C(\oO) \to C(\oO),\qquad F(\lambda,u):=\mK_\lambda (u_+^{p-1}),
\end{equation}
where $u_+(x):=\max\{u(x),0\}$ denotes the positive part of $u$ and $\mK_\lambda$ is defined in
\eqref{eq:def-K-lambda}. 

\begin{lemma}
\label{sec:glob-branch-posit-1}
The map $F$: $\R\times C(\oO) \to C(\oO)$ is compact and continuous. 
\end{lemma}

\begin{proof}
We first recall that the linear operator $\mK_\lambda: C(\oO) \to C(\oO)$ is compact for every $\lambda \in \R$ by Lemma~\ref{comp-strictly-positive}(a). We also claim that the operator-valued mapping 
\begin{equation}
  \label{eq:lambda-K-map}
\R \to \cL(C(\oO)),\qquad  \lambda \mapsto \mK_\lambda 
\end{equation}
is continuous. Indeed, for  $u\in C(\oO)$ with $\|u\|_\infty\leq 1$,  $\lambda\in\R$ and $\mu\in[\lambda-1,\lambda+1]$ we have
\begin{align*}
\|\mK_\lambda(u)&-\mK_\mu(u) \|_\infty
\leq  \sup_{x\in\oO}\int_{\oO}|\Psi_\lambda(x-y)-\Psi_\mu(x-y)|  |Q(y)| |u(y)| \, dy\\
&\le \|Q\|_\infty \sup_{x\in\oO} \int_{\oO}|\Psi_\lambda(x-y)-\Psi_\mu(x-y)| \, dy  \leq \|Q\|_\infty \int_{B_{{r_Q}}(0)}|\Psi_\lambda(z)-\Psi_\mu(z)|\, dz.
\end{align*}
Now, by Lemma~\ref{lem:bessel} (ii) and (v), 
$$
|\Psi_\lambda|, |\Psi_\mu|\leq 
\gamma_N\Psi_0+\zeta_N \left(|\lambda|+1\right)^{\frac{N-3}{4}}|\cdot |^{\frac{1-N}{2}} \in L^1(B_{{r_Q}}(0))
$$ 
and by Lemma~\ref{lem:bessel} (i), 
$$
\Psi_\mu(z) \to \Psi_\lambda(z)\qquad \text{as $\mu\to\lambda$ for all $z\neq 0$.}
$$
Hence, the dominated convergence theorem yields
$$
\left\|\mK_\lambda(u)-\mK_\mu(u)\right\|_\infty\to 0, \quad\text{ as $\mu\to\lambda$ uniformly for $\|u\|_\infty\leq 1$.}
$$
This yields the continuity of the map in \eqref{eq:lambda-K-map}. Finally, we recall that the map $u\mapsto u_+^{p-1}$ is Lipschitz continuous on bounded subsets of $C(\oO)$. From these properties, the claim follows. 
 \end{proof}

\begin{theorem}\label{thm:continuation}
Assume that one of the assumptions $(A1)$ or $(A2)$ from the introduction is satisfied. 
Then, for any $\Lambda^* \in (0,\lambda_Q)$,  there exists a global branch of positive solutions of \eqref{eq:K-equation} in 
$(-\infty, \Lambda^*] \times C(\oO)$.
\end{theorem}
\begin{proof}
In a first step, we fix $\Lambda_* >0$, and we 
show the existence of a global branch of positive solutions of \eqref{eq:K-equation} in 
$[-\Lambda_*,\Lambda^*] \times C(\oO)$. For this, we apply the Leray-Schauder continuation principle (see, e.g., \cite[Theorem 14.C]{zeidler}) to the operator
$F$ defined in \eqref{eq:F-operator-def} and the open annulus 
$$
\cA_{\delta,\gamma}:=\{u\in C(\oO)\, :\, \delta<\|u\|_\infty<\gamma\} \quad \subset \quad C(\oO),
$$
where $0<\delta<\gamma$ are suitably chosen.
We first need to ensure that, for all $\lambda\in[-\Lambda_*,{\Lambda^*}]$, the equation $u=F(\lambda,u)$ does
not have any solution on the boundary of $\cA_{\delta,\gamma}$. We start by remarking that every solution 
$u\in C(\oO)$ of the equation $u=F(\lambda,u)$ with $\lambda\leq{\Lambda^*}$ satisfies $u=u_+\geq 0$ in $\oO$ by Lemma~\ref{comp-strictly-positive}(b) since $\Lambda^\ast<\lambda_Q$. Letting $\delta_0>0$ be given by Lemma~\ref{sec:infinity-norm-lower-bounds}, we then see that either $u\equiv 0$
or the estimate
\begin{equation}\label{eqn:delta_0}
\|u\|_\infty\geq \delta_0
\end{equation}
holds. Indeed, this follows from an application of Lemma~\ref{sec:infinity-norm-lower-bounds} to the canonical extension of $u$ on $\R^N$ given by \eqref{eq:extension-formula} with $f = u^{p-1}$ (see the remarks before Lemma~\ref{comp-strictly-positive}).

Next, we set $u_{0,\lambda}= \mR_{\lambda}(Q)$ and 
use Propositions~\ref{prop:a_priori_w_restr} and \ref{prop:a_priori_w_all_p} to obtain $C>0$ such that
\begin{equation}\label{eqn:a_priori}
\|u\|_\infty+t \leq C \qquad \text{for all solutions of $u=F(\lambda,u)+t u_{0,\lambda}$ with $\lambda\in[-\Lambda_*,{\Lambda^*}]$.}
\end{equation}
Choosing $0<\delta<\delta_0$ and $\gamma> \max\{C,\delta\}$, we infer, in particular, that all solutions 
$u\in \overline{\cA_{\delta,\gamma}}$ of the equation $u=F(\lambda,u)$ with $\lambda\in[-\Lambda_*,{\Lambda^*}]$ are contained in the open set $\cA_{\delta,\gamma}$.

We now claim that the Leray-Schauder fixed-point index 
$\text{ind}(F(-\Lambda_*,\cdot),\cA_{\delta,\gamma})$ is different from zero (see e.g. \cite{zeidler} for a definition of this index).
To prove this,  we note that
$$
F(-\Lambda_*,u)\neq \tau u\text{ for all }\tau\geq 1\text{ and all }u\in C(\oO)\text{ with }\|u\|_\infty=\delta
$$
by Lemma~\ref{sec:infinity-norm-lower-bounds}. Hence, by  \cite[Theorem 13.A]{zeidler}), we find that  
$$
\text{ind}(F(-\Lambda_*,\cdot),B_\delta(0))=1. 
$$
Here and in the following, we put $B_\rho(0):= \{u \in C(\oO)\::\: \|u\|_\infty < \rho\}$ for $\rho>0$.
Considering the compact homotopy 
$$
H: [0,T]\times C(\oO) \to C(\oO),\qquad H(t,u)=F(-\Lambda_*,u)+t u_{0,-\Lambda_*},
$$
we also find that 
$$
u\neq H(t,u)\text{ for all }t\geq 0\text{ and all }u\in C(\oO)\text{ with }\|u\|_\infty=\gamma,
$$
since $\gamma>C$ and \eqref{eqn:a_priori} holds. In addition, choosing $T>C$, we have $u\neq H(T,u)$ for all $u\in C(\oO)$ again by \eqref{eqn:a_priori}. By the existence principle and the homotopy invariance of the fixed-point index,
we thus obtain 
$$
\text{ind}(F(-\Lambda_*,\cdot),B_\gamma(0))=\text{ind}(H(T,\cdot),B_\gamma(0))=0.
$$
Using the additivity of the fixed-point index, we conclude that 
$$
\text{ind}(F(-\Lambda_*,\cdot),\cA_{\delta,\gamma})=
\text{ind}(F(-\Lambda_*,\cdot),B_\gamma(0))-\text{ind}(F(-\Lambda_*,\cdot),B_\delta(0))=-1.
$$
Hence, all the assumptions of the Leray-Schauder continuation principle given in \cite[Theorem 14.C]{zeidler} are satisfied, and this principle yields the existence of a connected component of the set of nonnegative solutions $(\lambda,u)$ of \eqref{eq:K-equation} in $[-\Lambda_*,{\Lambda^*}]\times\cA_{\delta,\gamma}$ which intersects $\{-\Lambda_*\} \times \cA_{\delta,\gamma}$ and $\{{\Lambda^*}\} \times \cA_{\delta,\gamma}$. By Lemma~\ref{comp-strictly-positive}(b), it is a global branch of positive solutions of \eqref{eq:K-equation} in $[-\Lambda_*,{\Lambda^*}] \times C(\oO)$. This concludes the first step of the proof.\\
In a second step, we consider a strictly increasing sequences of numbers 
$\Lambda_{*,n}>0$ such that $\Lambda_{*,n} \to \infty$ as $n \to \infty$. As a consequence of what we have already proved, for every $n \in \N$ there exists a global branch $\cC_n$ of positive solutions of \eqref{eq:K-equation} in $[-\Lambda_{*,n},{\Lambda^{*}}] \times C(\oO)$. We now consider the subset $X$ of all points $(\lambda,u) \in (-\infty,\Lambda^*] \times C(\oO)$ such that 
$u$ is a positive solution of \eqref{eq:K-equation}. From Lemma~\ref{sec:infinity-norm-lower-bounds}, Lemma~\ref{comp-strictly-positive}, Propositions~\ref{prop:a_priori_w_restr} and \ref{prop:a_priori_w_all_p}, it then follows that 
\begin{equation}
\label{eq:-local-compact}
\text{$X \cap \bigl([r,\Lambda^*] \times C(\oO)\bigr)$ is compact for every $r\leq\Lambda^*$.}  
\end{equation}
In particular, $X$ is a locally compact metric space with the metric inherited from $(-\infty,\Lambda^*] \times C(\oO)$.  Moreover, after passing to a subsequence, we see that 
$X$ and the subsets $\cC_n \subset X$, $n \in \N$ satisfy the assumptions of Proposition~\ref{sec:appendix-1} from the Appendix. Indeed, let $z_n:= (\Lambda^*,u_n) \in \cC_n \subset X$ for $n \in \N$. Passing to a subsequence and using \eqref{eq:-local-compact}, we may then assume that $z_n \to z_* = (\Lambda^*,u_*)  \in X$. Moreover, since $\cC_n$ is a global branch in $[-\Lambda_{*,n},{\Lambda^{*}}] \times C(\oO)$ and $\Lambda_{*,n} \to \infty$ as $n \to \infty$, it follows that the sets $\bigcup_{n \ge m} \cC_n$, $m \in \N$ are not relatively compact in $X$. Consequently, Proposition~\ref{sec:appendix-1} implies that the connected component $\cC$ of $X$ which contains $z_*$ is not relatively compact. From \eqref{eq:-local-compact} and the fact that $z_* \in 
\{\Lambda^* \} \cap C(\oO)$, it then follows that $\cC$ is a global branch in $(-\infty, \Lambda^*] \times C(\oO)$, as desired.
\end{proof}

We conclude by completing the 

\begin{proof}[Proof of Theorem~\ref{main-theorem}]
Let $\cC$ be the global branch of positive solutions of \eqref{eq:K-equation} in $(-\infty,\Lambda^*] \times C(\oO)$ given by
Theorem~\ref{thm:continuation} and consider for each $(\lambda,u)\in \cC$ the canonical extension of $u$ given by \eqref{eq:extension-formula}
with $f=u^{p-1}$. Then $u\in L^p_{\text{loc}}(\R^N)$ solves \eqref{eqn:fp_u}. By Lemma~\ref{lem:reg1bis}(a) and Proposition~\ref{prop:l_infty}, $u$ is a strong solution of \eqref{eq:eq-general-schr-1}
in $W^{2,t}_{\text{loc}}(\R^N)\cap C^{1,\gamma}_{\text{loc}}(\R^N)$ for all $1\leq t<\infty$ and all $\gamma\in(0,1)$.
In addition, the positivity of $u$ on $\oO=\supp\, Q$ implies, together with Lemma~\ref{lem:bessel} (ii),
that $u=\Psi_\lambda\ast(Qu^{p-1})>0$ on $\R^N$, in the case where $\lambda\leq 0$.

Let now $s>\frac{2N}{N-1}$. To prove that $\cC$ is connected in $\R\times L^s(\R^N)$,
it suffices to show that 
\begin{equation}
  \label{eq:continuity-final-0}
\text{the map $\quad \R\times L^\infty(\Omega) \to L^s(\R^N)$, $\quad (\lambda,u)\mapsto \Psi_\lambda\ast(Q|u|^{p-2}u)\quad$ is continuous.}  
\end{equation}
Here the convolution is understood as in \eqref{eq:extension-formula}. Since the map $L^\infty(\Omega) \to L^\infty(\Omega), \: u \mapsto |u|^{p-2}u$ is continuous, \eqref{eq:continuity-final-0} follows once we have shown that 
\begin{equation}
  \label{eq:continuity-final}
\text{the map $\quad \R\times L^\infty(\R^N) \to L^s(\R^N)$, $\quad (\lambda,f)\mapsto \Psi_\lambda\ast(Qf)\quad$ is continuous.}  
\end{equation}
For $\lambda\in\R$ fixed we can find, as a consequence of Lemma~\ref{lem:bessel} (ii) and (v), constants $c_1, c_2>0$
such that
$$
|\Psi_\lambda(x)|\leq c_1|x|^{2-N}1_{B_1(0)}(x) + c_2|x|^{\frac{1-N}2}1_{\R^N\backslash B_1(0)}(x), \quad x\in\R^N.
$$
Therefore, the weak Young inequality gives
\begin{align*}
\|\Psi_\lambda &\ast (Qf)\|_s \leq c_1\|\, |\cdot|^{2-N}\|_{\frac{N}{N-2},w}\|Qf\|_{t_1}
+c_2\|\, |\cdot|^{\frac{1-N}{2}}\|_{\frac{2N}{N-1},w} \|Qf\|_{t_2}\\
&\leq \Bigl\{c_1\|\, |\cdot|^{2-N}\|_{\frac{N}{N-2},w} |\Omega|^{\frac1{t_1}}
+c_2\|\, |\cdot|^{\frac{1-N}{2}}\|_{\frac{2N}{N-1},w} |\Omega|^{\frac1{t_2}}\Bigr\}\|Q\|_\infty \|f\|_\infty
=:C(\lambda)\|f\|_\infty,
\end{align*}
where $\frac1{t_1}=\frac1s+\frac2N$ and $\frac1{t_2}=\frac1s+\frac{N+1}{2N}$. Notice that $t_1, t_2>1$, since
by assumption $s>\frac{2N}{N-1}$. Therefore, the linear map $L^\infty (\R^N) \to L^s(\R^N)$, 
$f\mapsto \Psi_\lambda\ast(Qf)$ is continuous for every $\lambda\in\R$.
On the other hand, using Young's inequality, we find
\begin{align*}
\|(\Psi_\mu-\Psi_\lambda)\ast(Qf)\|_s &\leq \|(\Psi_\mu-\Psi_\lambda)1_{B_1(0)}\|_1 \|Qf\|_s
+\|(\Psi_\mu-\Psi_\lambda)1_{\R^N\backslash B_1(0)}\|_s\|Qf\|_1\\
&\leq \Bigl\{\|\Psi_\mu-\Psi_\lambda\|_{L^1(B_1(0))} \|Q\|_s 
+\|\Psi_\mu-\Psi_\lambda\|_{L^s(\R^N\backslash B_1(0))}\|Q\|_1\Bigr\}\|f\|_\infty.
\end{align*}
Using Lemma~\ref{lem:bessel} and the dominated convergence theorem in the same way as Lemma~\ref{sec:glob-branch-posit-1},
it follows that $\|\Psi_\mu-\Psi_\lambda\|_{L^1(B_1(0))}\to 0$ and 
$\|\Psi_\mu-\Psi_\lambda\|_{L^s(\R^N\backslash B_1(0))}\to 0$ as $\mu\to\lambda$. 
Therefore, the map $\R\times L^\infty(\R^N)$ $\to$ $L^s(\R^N)$, $(\lambda,f)\mapsto\Psi_\lambda\ast(Qf)$
is continuous in $\lambda$, uniformly on bounded sets of $L^\infty(\R^N)$. Hence, writing
\begin{align*}
\|\Psi_\lambda\ast(Qf)-\Psi_\mu\ast(Qg)\|_s
&\leq \|\Psi_\lambda\ast(Qf-Qg)\|_s+\|(\Psi_\lambda-\Psi_\mu)\ast(Qg)\|_s,
\end{align*}
and using the above estimates, we obtain that 
$$\Psi_\mu\ast(Qg)\to \Psi_\lambda\ast(Qf) \qquad \text{in $L^s(\R^N)$ 
as $(\mu,g)\to(\lambda,f)$ in $\R\times L^\infty(\R^N)$,}
$$
which gives the continuity of the map in \eqref{eq:continuity-final}. 
As a consequence, for fixed $\Lambda^\ast \in (0,\lambda_Q)$, the global branch $\cC$ of Theorem~\ref{thm:continuation} translates to a global branch of solutions 
of \eqref{eq:eq-general-schr-1} in $(-\infty,\Lambda^\ast] \times L^s(\R^N)$  with the properties (i) and (ii) 
of Theorem~\ref{main-theorem}. 
In order to prove the property (iii), we show that every nontrivial solution $u\in L^p_{\text{loc}}(\R^N)$ of \eqref{eqn:fp_u} with $Q$ as in Theorem~\ref{main-theorem} and $0<\lambda<\lambda_Q$ changes sign on $\R^N$. 
Suppose by contradiction that for some $0<\lambda<\lambda_Q$ there exists a nontrivial nonnegative solution $u\in L^p_{\text{loc}}(\R^N)$ of \eqref{eqn:fp_u}. By Proposition~\ref{prop:l_infty}, $u\in W^{2,t}_{\text{loc}}(\R^N)\cap C^{1,\gamma}_{\text{loc}}(\R^N)$ for all $1\leq t<\infty$ and all $0<\gamma<1$. Therefore, by the strong maximum principle,
$u$ is a strong positive solution of \eqref{eq:eq-general-schr-1}. 
In particular, $\Delta u+\lambda u\leq 0$ on $\R^N$. Consider the function 
$$
\psi\in C^\infty(\R^N),\qquad \psi(x)=|x|^{\frac{2-N}{2}}J_{\frac{N-2}{2}}(\sqrt{\lambda}|x|), \quad \text{for $x\in\R^N$,}
$$
where $J_{\frac{N-2}{2}}$ denotes the Bessel function of the first kind of order $\frac{N-2}{2}$. This function satisfies $\Delta\psi+\lambda\psi=0$ on $\R^N$ and $\psi(x)=O(|x|^{\frac{1-N}{2}})$ as $|x|\to\infty$. By a result of Berestycki, Caffarelli and Nirenberg \cite[Theorem 1.8]{BCN97}, 
there is a constant $C\in\R$ such that $\psi=Cu$, but this is impossible
since $\psi$ changes sign on $\R^N$. This contradiction shows that every nontrivial solution of \eqref{eqn:fp_u}
with $0<\lambda<\lambda_Q$ must change sign on $\R^N$ and therefore the proof of Theorem~\ref{main-theorem} is complete.
\end{proof}

\section*{Appendix}
\renewcommand{\thesection}{A}\label{sec:appendix}
\setcounter{theorem}{0}
In this section we add a result from general topology which has been used in the proof of Theorem~\ref{thm:continuation}. It is a variant of a classical Lemma by Whyburn (see \cite[Theorem (9.1)]{whyburn}). For similar variants which have inspired the following proposition, see \cite{ma-an09,li-sun09}. 

\begin{proposition}
\label{sec:appendix-1}
Let $(X,d)$ be a locally compact metric space, $z_* \in X$, and let $\cC_n \subset X$, $n \in \N$ be connected subsets satisfying the following assumptions. 
\begin{itemize}
\item[(i)] There exist points $z_n \in \cC_n$, $n \in \N$ such that $z_n \to z_* \in X$ as $n \to \infty$.
\item [(ii)] The sets $\bigcup_{n \ge m} \cC_n$, $m \in \N$ are not relatively compact in $(X,d)$.
\end{itemize}
Then the connected component $\cC \subset X$ of $X$ which contains $z_*$ is not relatively compact. 
\end{proposition}

For the proof of this proposition, we need the following well-known result (see \cite{whyburn}).

\begin{lemma}
\label{sec:whyburn-1}  
Suppose that $(X,d)$ is a compact metric space, $A$ and $B$ are disjoint closed subsets of $X$, and suppose that no  connected component of $X$ intersects both $A$ and $B$. Then there exist two disjoint compact subsets $X_A, X_B \subset X$ such that $A \subset X_A$,  $B \subset X_B$ and $X = X_A \cup X_B$. 
\end{lemma}

\begin{proof}[Proof of Proposition~\ref{sec:appendix-1}]
We suppose by contradiction that $\cC$ is relatively compact. 
Since, by definition, $\cC \subset X$ is closed, it follows that $\cC$ is compact. 
Since $X$ is locally compact, there exists a compact neighborhood $V \subset X$ of the set $\cC$. 
Then $\cC$ and $\partial V$ are non-intersecting closed subsets contained in the compact metric space $(V,d)$, and the maximal connectedness of $\cC$ implies that that there does not exist a connected component of $V$ which intersects 
$\cC$ and $\partial V$. 
By Lemma~\ref{sec:whyburn-1}, there exist disjoint compact subsets $X_A, X_B \subset V$ such that $\cC \subset X_A$,  $\partial V \subset X_B$ and 
\begin{equation}
  \label{eq:union-A-B}
V = X_A \cup X_B.  
\end{equation}
We may then choose a compact neighborhood $V_1 \subset X$ of $X_A$ such that 
$V_1 \cap X_B= \varnothing$, and we consider the compact set $V_2= V \cap V_1$. We have 
$$
\partial V_2 \subset 
[\partial V \cap V_1] \cup [\partial V_1 \cap  V] \subset V \cap V_1, 
$$
and thus it follows that 
$$
\partial V_2 \cap X_A= \varnothing = \partial V_2 \cap X_B.
$$
Consequently, 
\begin{equation}
  \label{eq:empty-intersect}
\partial V_2 = \varnothing  
\end{equation}
by \eqref{eq:union-A-B}, which implies in particular that $V_2$ is also open in $X$. 
On the other hand, since $z_* \in \cC \subset X_A \subset V_2$, there exists $n_0 \in \N$ such that 
$z_n \in V_2$ for $n \ge n_0$, which means that $\cC_n \cap V_2 \not = \varnothing$ for $n \ge n_0$. 
The connectedness of $\cC_n$ and \eqref{eq:empty-intersect} then imply that 
$$
\cC_n \subset V_2 \qquad  \text{for $n \ge n_0$,}
$$
but this contradicts assumption (ii) since $V_2$ is compact. Thus the proof is finished. 
\end{proof}

\section*{Acknowledgements}
\noindent This research was supported by Grant WE 2821/5-1 of the Deutsche Forschungsgemeinschaft.


\bibliographystyle{abbrv}

\begin{thebibliography}{10}
	
	\bibitem{amann-lopez98}
	H.~Amann and J.~L{\'o}pez-G{\'o}mez.
	\newblock A priori bounds and multiple solutions for superlinear indefinite
	elliptic problems.
	\newblock {\em J. Differential Equations}, 146:336--374, 1998.
	
	\bibitem{AB98.1}
	A.~Ambrosetti and M.~Badiale.
	\newblock Homoclinics: {P}oincar\'e-{M}elnikov type results via a variational
	approach.
	\newblock {\em Ann. Inst. H. Poincar\'e Anal. Non Lin\'eaire}, 15(2):233--252,
	1998.
	
	\bibitem{AB98}
	A.~Ambrosetti and M.~Badiale.
	\newblock Variational perturbative methods and bifurcation of bound states from
	the essential spectrum.
	\newblock {\em Proc. Roy. Soc. Edinburgh Sect. A}, 128(6):1131--1161, 1998.
	
	\bibitem{BP04}
	M.~Badiale and A.~Pomponio.
	\newblock Bifurcation results for semilinear elliptic problems in {$\mathbb{R}^N$}.
	\newblock {\em Proc. Roy. Soc. Edinburgh Sect. A}, 134(1):11--32, 2004.
	
	\bibitem{BCN97}
	H.~Berestycki, L.~Caffarelli, and L.~Nirenberg.
	\newblock Further qualitative properties for elliptic equations in unbounded
	domains.
	\newblock {\em Ann. Scuola Norm. Sup. Pisa Cl. Sci. (4)}, 25(1-2):69--94, 1997.
	\newblock Dedicated to Ennio De Giorgi.
	
	\bibitem{berestycki-cd-nirenberg94}
	H.~Berestycki, I.~Capuzzo-Dolcetta, and L.~Nirenberg.
	\newblock Superlinear indefinite elliptic problems and nonlinear {L}iouville
	theorems.
	\newblock {\em Topol. Methods Nonlinear Anal.}, 4(1):59--78, 1994.
	
	\bibitem{berestycki-lions83}
	H.~Berestycki and P.-L. Lions.
	\newblock Nonlinear scalar field equations. {I}. {E}xistence of a ground state.
	\newblock {\em Arch. Rational Mech. Anal.}, 82(4):313--345, 1983.
	
	\bibitem{berestycki-lions83b}
	H.~Berestycki and P.-L. Lions.
	\newblock Nonlinear scalar field equations. {II}. {E}xistence of infinitely
	many solutions.
	\newblock {\em Arch. Rational Mech. Anal.}, 82(4):347--375, 1983.
	
	\bibitem{brezis-turner:77}
	H.~Br{\'e}zis and R.~E.~L. Turner.
	\newblock On a class of superlinear elliptic problems.
	\newblock {\em Comm. Partial Differential Equations}, 2(6):601--614, 1977.
	
	\bibitem{Ca95}
	D.~M. Cao.
	\newblock Eigenvalue problems and bifurcation of semilinear elliptic equation
	in {${\bf R}^N$}.
	\newblock {\em Nonlinear Anal.}, 24(4):529--554, 1995.
	
	\bibitem{dancer:73}
	E.~N. Dancer.
	\newblock On the structure of solutions of non-linear eigenvalue problems.
	\newblock {\em Indiana Univ. Math. J.}, 23:1069--1076, 1973/74.
	
	\bibitem{dancer:2002}
	E.~N. Dancer.
	\newblock Bifurcation from simple eigenvalues and eigenvalues of geometric
	multiplicity one.
	\newblock {\em Bull. London Math. Soc.}, 34(5):533--538, 2002.
	
	\bibitem{ding.ni:86}
	W.~Y. Ding and W.-M. Ni.
	\newblock On the existence of positive entire solutions of a semilinear
	elliptic equation.
	\newblock {\em Arch. Rational Mech. Anal.}, 91(4):283--308, 1986.
	
	\bibitem{du.li:05}
	Y.~Du and S.~Li.
	\newblock Nonlinear {L}iouville theorems and a priori estimates for indefinite
	superlinear elliptic equations.
	\newblock {\em Adv. Differential Equations}, 10(8):841--860, 2005.
	
	\bibitem{evequoz:2015-1}
	G.~Ev{\'e}quoz.
	\newblock A dual approach in {O}rlicz spaces for the nonlinear {H}elmholtz
	equation.
	\newblock {\em Z. Angew. Math. Phys.}, 66(6):2995--3015, 2015.
	
	\bibitem{evequoz-weth14}
	G.~Ev{\'e}quoz and T.~Weth.
	\newblock Real solutions to the nonlinear {H}elmholtz equation with local
	nonlinearity.
	\newblock {\em Arch. Rat. Mech. Anal.}, 211(2):359--388, 2014.
	
	\bibitem{evequoz-weth-dual}
	G.~Ev{\'e}quoz and T.~Weth.
	\newblock Dual variational methods and nonvanishing for the nonlinear
	{H}elmholtz equation.
	\newblock {\em Adv. in Math.}, 280:690--728, 2015.
	
	\bibitem{floer.weinstein:86}
	A.~Floer and A.~Weinstein.
	\newblock Nonspreading wave packets for the cubic {S}chr\"odinger equation with
	a bounded potential.
	\newblock {\em J. Funct. Anal.}, 69(3):397--408, 1986.
	
	\bibitem{gelfand}
	I.~M. Gel{$'$}fand and G.~E. Shilov.
	\newblock {\em Generalized functions. {V}ol. 1}.
	\newblock Academic Press [Harcourt Brace Jovanovich Publishers], New York, 1964
	[1977].
	\newblock Properties and operations, Translated from the Russian by Eugene
	Saletan.
	
	\bibitem{Gi98}
	J.~Giacomoni.
	\newblock Global bifurcation results for semilinear elliptic problems in
	{$\mathbf{R}^N$}.
	\newblock {\em Comm. Partial Differential Equations}, 23(11-12):1875--1927,
	1998.
	
	\bibitem{gidas.spruck81}
	B.~Gidas and J.~Spruck.
	\newblock A priori bounds for positive solutions of nonlinear elliptic
	equations.
	\newblock {\em Comm. Partial Differential Equations}, 6(8):883--901, 1981.
	
	\bibitem{He86}
	H.-P. Heinz.
	\newblock Nodal properties and bifurcation from the essential spectrum for a
	class of nonlinear {S}turm-{L}iouville problems.
	\newblock {\em J. Differential Equations}, 64(1):79--108, 1986.
	
	\bibitem{KRS87}
	C.~E. Kenig, A.~Ruiz, and C.~D. Sogge.
	\newblock Uniform {S}obolev inequalities and unique continuation for second
	order constant coefficient differential operators.
	\newblock {\em Duke Math. J.}, 55(2):329--347, 1987.
	
	\bibitem{Ku79}
	T.~K{\"u}pper.
	\newblock The lowest point of the continuous spectrum as a bifurcation point.
	\newblock {\em J. Differential Equations}, 34(2):212--217, 1979.
	
	\bibitem{lebedev}
	N.~N. Lebedev.
	\newblock {\em Special functions and their applications}.
	\newblock Dover Publications Inc., New York, 1972.
	\newblock Revised edition, translated from the Russian and edited by Richard A.
	Silverman, Unabridged and corrected republication.
	
	\bibitem{li-sun09}
	H.~Li and J.~Sun.
	\newblock Positive solutions of sublinear {S}turm-{L}iouville problems with
	changing sign nonlinearity.
	\newblock {\em Comput. Math. Appl.}, 58(9):1808--1815, 2009.
	
	\bibitem{ma-an09}
	R.~Ma and Y.~An.
	\newblock Global structure of positive solutions for nonlocal boundary value
	problems involving integral conditions.
	\newblock {\em Nonlinear Anal.}, 71(10):4364--4376, 2009.
	
	\bibitem{magnus:88}
	R.~J. Magnus.
	\newblock On perturbations of a translationally-invariant differential
	equation.
	\newblock {\em Proc. Roy. Soc. Edinburgh Sect. A}, 110(1-2):1--25, 1988.
	
	\bibitem{polacik.quittner:05}
	P.~Pol{\'a}{\v{c}}ik and P.~Quittner.
	\newblock Liouville type theorems and complete blow-up for indefinite
	superlinear parabolic equations.
	\newblock In {\em Nonlinear elliptic and parabolic problems}, volume~64 of {\em
		Progr. Nonlinear Differential Equations Appl.}, pages 391--402. Birkh\"auser,
	Basel, 2005.
	
	\bibitem{rabinowitz71}
	P.~H. Rabinowitz.
	\newblock Some global results for nonlinear eigenvalue problems.
	\newblock {\em J. Functional Analysis}, 7:487--513, 1971.
	
	\bibitem{rabinowitz:92}
	P.~H. Rabinowitz.
	\newblock On a class of nonlinear {S}chr\"odinger equations.
	\newblock {\em Z. Angew. Math. Phys.}, 43(2):270--291, 1992.
	
	\bibitem{ramosetal98}
	M.~Ramos, S.~Terracini, and C.~Troestler.
	\newblock Superlinear indefinite elliptic problems and {P}oho\v zaev type
	identities.
	\newblock {\em J. Funct. Anal.}, 159(2):596--628, 1998.
	
	\bibitem{Ro89}
	W.~Rother.
	\newblock Bifurcation of nonlinear elliptic equations on {${\bf R}^N$}.
	\newblock {\em Bull. London Math. Soc.}, 21(6):567--572, 1989.
	
	\bibitem{stuart:80-0}
	C.~A. Stuart.
	\newblock Bifurcation for variational problems when the linearisation has no
	eigenvalues.
	\newblock {\em J. Funct. Anal.}, 38(2):169--187, 1980.
	
	\bibitem{stuart:80}
	C.~A. Stuart.
	\newblock Bifurcation from the continuous spectrum in the {$L^2$}-theory of
	elliptic equations on {${\bf R}^n$}.
	\newblock In {\em Recent methods in nonlinear analysis and applications
		({N}aples, 1980)}, pages 231--300. Liguori, Naples, 1981.
	
	\bibitem{stuart:82}
	C.~A. Stuart.
	\newblock Bifurcation for {D}irichlet problems without eigenvalues.
	\newblock {\em Proc. London Math. Soc. (3)}, 45(1):169--192, 1982.
	
	\bibitem{stuart:85}
	C.~A. Stuart.
	\newblock A global branch of solutions to a semilinear equation on an unbounded
	interval.
	\newblock {\em Proc. Roy. Soc. Edinburgh Sect. A}, 101(3-4):273--282, 1985.
	
	\bibitem{stuart:88}
	C.~A. Stuart.
	\newblock Bifurcation in {$L^p({\bf R}^N)$} for a semilinear elliptic equation.
	\newblock {\em Proc. London Math. Soc. (3)}, 57(3):511--541, 1988.
	
	\bibitem{stuart:89}
	C.~A. Stuart.
	\newblock Bifurcation of homoclinic orbits and bifurcation from the essential
	spectrum.
	\newblock {\em SIAM J. Math. Anal.}, 20(5):1145--1171, 1989.
	
	\bibitem{stuart97}
	C.~A. Stuart.
	\newblock Bifurcation from the essential spectrum.
	\newblock In {\em Topological nonlinear analysis, {II} ({F}rascati, 1995)},
	volume~27 of {\em Progr. Nonlinear Differential Equations Appl.}, pages
	397--443. Birkh\"auser Boston, Boston, MA, 1997.
	
	\bibitem{toland:1984}
	J.~F. Toland.
	\newblock Positive solutions of nonlinear elliptic equations---existence and
	nonexistence of solutions with radial symmetry in {$L_{p}({\bf R}^{N})$}.
	\newblock {\em Trans. Amer. Math. Soc.}, 282(1):335--354, 1984.
	
	\bibitem{watson}
	G.~N. Watson.
	\newblock {\em A treatise on the theory of {B}essel functions}.
	\newblock Cambridge Mathematical Library. Cambridge University Press,
	Cambridge, 1995.
	\newblock Reprint of the second (1944) edition.
	
	\bibitem{weth05}
	T.~Weth.
	\newblock Global bifurcation branches for radially symmetric {S}chr\"odinger
	equations.
	\newblock {\em Adv. Differential Equations}, 10(7):721--746, 2005.
	
	\bibitem{whyburn}
	G.~T. Whyburn.
	\newblock {\em Topological analysis}.
	\newblock Second, revised edition. Princeton Mathematical Series, No. 23.
	Princeton University Press, Princeton, N.J., 1964.
	
	\bibitem{zeidler}
	E.~Zeidler.
	\newblock {\em Nonlinear functional analysis and its applications. {I}
		Fixed-point theorems}.
	\newblock Springer-Verlag, New York, 1986.
	
	\bibitem{ZZ88}
	X.~P. Zhu and H.~S. Zhou.
	\newblock Bifurcation from the essential spectrum of superlinear elliptic
	equations.
	\newblock {\em Appl. Anal.}, 28(1):51--66, 1988.
	
\end{thebibliography}

\end{document}